\documentclass[12pt]{amsart}

\usepackage{a4wide}
\usepackage{amssymb}
\usepackage{amsmath}
\usepackage{amsfonts}
\usepackage{amsthm}
\usepackage{enumerate}
\usepackage{hyperref}
\usepackage[usenames,dvipsnames]{xcolor}
\numberwithin{equation}{section}

\newtheorem{theorem}{Theorem}
\numberwithin{theorem}{section}
\newtheorem{corollary}[theorem]{Corollary}
\newtheorem{proposition}[theorem]{Proposition}
\newtheorem{lemma}[theorem]{Lemma}

\theoremstyle{definition}
\newtheorem{definition}[theorem]{Definition}

\newtheorem*{remark*}{Remark}
\newtheorem{example}[theorem]{Example}

\newcommand{\LBo}[1]{\Delta_{#1}^*} 

\newcommand{\BeOp}[1]{{\mathbf{J}_{#1}}}
\newcommand{\BeKe}[1]{{B^{(#1)}}}
\newcommand{\BeKeMod}[1]{{\mathcal{B}^{(#1)}}} 
\newcommand{\BeKeNumIntErr}[2]{{\mathcal{B}_{#1}^{(#2)}}} 
\newcommand{\filBeKe}[2]{{B_{#1}^{(#2)}}}

\newcommand{\sph}[1]{\mathbb{S}^{#1} 
}

\newcommand{\sob}[3]{ 
\mathbb{W}_{#1}^{#2}(\mathbb{S}^{#3})}

\newcommand{\Lp}[2]{ 
\mathbb{L}_{#1}(\mathbb{S}^{#2})}

\newcommand{\MESHRATIO}{\gamma}

\newcommand{\Eb}[3]{ E_{#1}(#2)_{#3} } 

\newcommand{\filterpsi}[2]{ \psi_{#1}^{#2} }

\newcommand{\R}{\mathbb{R}}

\newcommand{\DEF}{{\,:=\,}}

\newcommand{\PT}[1]{\mathbf{#1}}

\newcommand{\re}{\mathop{\mathrm{Re}}}

\newcommand{\bsx}{\mathbf{x}}

\DeclareMathOperator{\dd}{\mathrm{d}}

\DeclareMathOperator{\xctint}{I}
\DeclareMathOperator{\numint}{Q}
\DeclareMathOperator{\ClausenCi}{Ci}
\DeclareMathOperator{\ClausenSi}{Si}
\DeclareMathOperator{\COV}{\rho}
\DeclareMathOperator{\gammafcn}{\Gamma}
\DeclareMathOperator{\SEP}{\delta}
\DeclareMathOperator{\WCE}{wce}
\DeclareMathOperator{\zetafcn}{\zeta}

\DeclareMathOperator{\HyperF}{F}
\newcommand{\Hypergeom}[5]{{\sideset{_#1}{_#2}\HyperF\!\left(\substack{\displaystyle#3\\\displaystyle#4};#5\right)}}

\newcommand{\Pochhsymb}[2]{{\left(#1\right)_{#2}}}

\title[Covering of spheres and equal weight cubature]{Covering of spheres by spherical caps and worst-case error for equal weight cubature in Sobolev spaces}

\author[J.~S.~Brauchart, J.~Dick, E.~B.~Saff, I.~H.~Sloan, Y.G.~Wang and R.~S.~Womersley]{J.~S.~Brauchart\textasteriskcentered{}, J.~Dick, E.~B.~Saff, I.~H.~Sloan, Y.G.~Wang and R.~S.~Womersley}
\thanks{\noindent \textasteriskcentered{} corresponding author. \\ This research was supported under Australian Research Council's Discovery Projects funding scheme (project number DP120101816). The research of the first author was also partially supported by the Austrian Science Fund FWF project F5510 (part of the Special Research Program (SFB) ``Quasi-Monte Carlo Methods: Theory and Applications''. 
The research of the third author was also supported by U.S. National Science Foundation grant DMS-1109266. }

\date{\today}

\date{\today}

\begin{document}

\address{J.~S.~Brauchart: Institut f\"ur Analysis und Computational Number Theory,
Graz University of Technology ,
Steyrergasse 30,
8010 Graz,
Austria}
\email{j.brauchart@tugraz.at}

\address{J.~S.~Brauchart, J.~Dick, I.~H.~Sloan, Y.~G.~Wang and R.~S.~Womersley: \linebreak
School of Mathematics and Statistics,
University of New South Wales,
Sydney, NSW, 2052,
Australia }
\email{josef.dick@unsw.edu.au, i.sloan@unsw.edu.au, yuguang.e.wang@gmail.com, \linebreak r.womersley@unsw.edu.au}

\address{E.~B.~Saff:
Center for Constructive Approximation,
Department of Mathematics, \linebreak
Vanderbilt University,
Nashville, TN 37240,
USA}
\email{edward.b.saff@vanderbilt.edu}

\begin{abstract}
We prove that the covering radius of an $N$-point subset $X_N$ of the unit sphere $\mathbb{S}^d \subset \mathbb{R}^{d+1}$ is bounded above by a power of the worst-case error for equal weight cubature $\frac{1}{N}\sum_{\mathbf{x} \in X_N}f(\mathbf{x}) \approx \int_{\mathbb{S}^d} f \, \mathrm{d} \sigma_d$ for functions  in the Sobolev space $\mathbb{W}_p^s(\mathbb{S}^d)$, where $\sigma_d$ denotes normalized area measure on $\mathbb{S}^d.$ These bounds are close to optimal when $s$ is close to $d/p$.
Our study of the worst-case error  along with results of Brandolini et al. motivate the definition of Quasi-Monte Carlo (QMC) design sequences  for $\mathbb{W}_p^s(\mathbb{S}^d)$, which have previously been introduced only in the Hilbert space setting $p=2$. We say that a sequence $(X_N)$ of $N$-point configurations is a QMC-design sequence for $\mathbb{W}_p^s(\mathbb{S}^d)$ with $s > d/p$ provided the worst-case equal weight cubature error for $X_N$ has order $N^{-s/d}$ as $N \to \infty$, a property that holds, in particular, for a sequence of spherical $t$-designs in which each design has order $t^d$ points. For the case $p = 1$, we deduce that any QMC-design sequence $(X_N)$ for $\mathbb{W}_1^s(\mathbb{S}^d)$ with $s > d$ has the optimal covering property; i.e., the covering radius of $X_N$ has order $N^{-1/d}$ as $N \to \infty$.

A significant portion of our effort is devoted to the formulation of the worst-case error in terms of a Bessel kernel, and showing that this kernel satisfies a Bernstein type inequality involving the mesh ratio of $X_N$. As a consequence we prove that any QMC-design sequence for $\mathbb{W}_p^s(\mathbb{S}^d)$ is also a QMC-design sequence for $\mathbb{W}_{p^\prime}^s(\mathbb{S}^d)$ for all $1 \leq p < p^\prime \leq \infty$ and, furthermore, if $(X_N)$ is a quasi-uniform QMC-design sequence for $\mathbb{W}_p^s(\mathbb{S}^d)$, then it is also a QMC-design sequence for $\mathbb{W}_p^{s^\prime}(\mathbb{S}^d)$ for all $s > s^\prime > d/p$.
\end{abstract}


\maketitle


\section{Introduction}

In this paper we consider covering the unit sphere $\sph{d}$ in $\R^{d+1}$, $d \geq 1$, with equal sized spherical caps, and establish a connection to equal weight cubature formulas that use the centers of those caps as sampling points for the function.  As a corollary, we will show that the optimal order of convergence of the worst-case equal weight cubature error for functions in a suitable Sobolev space implies asymptotically an optimal covering property by spherical caps.

\subsection*{Equal-weight numerical integration}
In the literature equal weight cubature is often given the name Quasi-Monte Carlo (see Niederreiter~\cite{Nie92} for the case of the unit cube). Thus a \emph{Quasi-Monte Carlo (QMC) method} is an equal weight numerical integration formula with \emph{deterministic} node set  in contrast to Monte Carlo methods: for a node set $X_N = \{ \PT{x}_1, \dots, \PT{x}_N \} \subset \sph{d}$, the QMC method
%
\begin{equation*}
\numint[X_N](f) \DEF \frac{1}{N} \sum_{k=1}^N f(\PT{x}_k)
\end{equation*}
is a natural approximation of the integral
\begin{equation*}
\xctint(f) \DEF \int_{\sph{d}} f( \PT{x} ) \dd \sigma_d( \PT{x} )
\end{equation*}
of a given continuous real-valued function $f$ on $\sph{d}$ with respect to the normalized surface area measure on $\sph{d}$. A node set $X_N$ is deterministically chosen in a sensible way so as to guarantee ``small'' error of numerical integration for functions in suitable subfamilies of the class of continuous functions $C(\sph{d})$. 

A fundamental example of such node sets are \emph{spherical $t$-designs\footnote{The symbol $X_N$ is used for general sets of $N$ points on $\mathbb{S}^d$, while $Z_{N_t}$ always refers to a spherical $t$-design with $N_t$ points.} $Z_{N_t} \subset \sph{d}$}, $t \geq 1$, introduced in \cite{DeGoSei1977}. They  define QMC methods that integrate exactly all spherical polynomials of degree $\leq t$:
\begin{equation} \label{eq:spherical.designs}
\numint[Z_{N_t}](P) = \xctint(P), \qquad \deg P \leq t.
\end{equation}
Thus, spherical $t$-designs yield zero error on polynomial subfamilies of $C(\sph{d})$.
The definition of spherical $t$-designs says nothing about the number of points $N_t$ that might be needed. A lower bound on $N_t$ of order $t^d$ was given in \cite{DeGoSei1977}. Recently, Bondarenko et al.~\cite{BoRaVi2013} proved: 
\begin{proposition} \label{prop:Bondarenko.et.al.A}
There exists $c_d>0$ such that to \emph{every} $N \geq c_d \, t^d$ and $t \geq 1$ there exists an $N$-point spherical $t$-design on $\sph{d}$.
\end{proposition}
\noindent 
This key result ensures that spherical $t$-designs with $N_t$ points of exactly the optimal order $t^d$ exist for every $t \geq 1$ (we write $N_t \asymp t^d$). A sequence $(Z_{N_t})$ of such designs with  optimal order for the number of points has the remarkable property, see \cite{BrHe2007,HeSl2005a}, that 
\begin{equation*}
| \numint[Z_{N_t}](f) - \xctint(f) | \leq c \, N_t^{-s/d} \, \| f \|_{H^s}
\end{equation*}
for all functions $f$ in a Sobolev space $H^s$ with smoothness index $s > d/2$ and norm $\| \cdot \|_{H^s}$ in the Hilbert space setting. The order of $N_t$ cannot be improved, see~\cite{He2006,HeSl2005b}. 
This observation motivated the introduction of \emph{QMC-design sequences} for Sobolev spaces~$H^s$ in~\cite{BrSaSlWo2014}: these are sequences of $N$-point sets that have the same error behavior as spherical $t$-designs, but with no polynomial exactness requirement.  
One purpose of this paper is to provide the extension to general Sobolev spaces. 

\subsection*{Covering for the sphere}
For a finite set  $X_{N} = \{\PT{x}_1, \PT{x}_2,\ldots, \PT{x}_N\} \subset
\sph{d}$ \emph{the covering radius} (or mesh norm, or fill radius) is defined
by
\begin{equation}\label{eq:covering_def}
\COV(X_N) \DEF \max_{\PT{x} \in \sph{d}} \min_{1 \leq k \leq N} \arccos(\PT{x} \cdot \PT{x}_k).
\end{equation}
Thus the covering radius is the geodesic radius of the largest hole in the
mesh formed by the point set $X_N$.  Equivalently, it is the minimal radius of
equal-sized spherical caps centered at the points of $X_N$ that cover $\sph{d}$.
There is a trivial lower bound on $\COV(X_N)$ arising from the fact that a
spherical cap of geodesic radius $\COV(X_N)$ has a surface measure of exact order
$[\COV(X_N)]^d$: it follows that there exists $c_d>0$ such that
\begin{equation*}
\COV(X_N) \geq c_d \, N^{-1/d} \qquad \text{for all $N$.}
\end{equation*}
We will therefore say that a sequence $(X_N)$ of point sets on $\sph{d}$ has the \emph{optimal covering property} if
\begin{equation} \label{eq:optimal.covering.property}
\COV(X_N) = \mathcal{O}( N^{-1/d} ) \qquad \text{as $N \to \infty$.}
\end{equation}

Yudin~\cite{Yud1995} showed that if $Z_{N_t}$ is a spherical $t$-design,
then $Z_{N_t}$  gives a covering of the sphere~$\sph{d}$ with radius $\eta_{t,d}$, where
$\cos(\eta_{t,d})$ is the largest zero of a certain Jacobi polynomial.
Reimer~\cite{Rei1999,Rei2000} extended Yudin's result to any positive weight
cubature rule that is exact for polynomials of degree at most $t$,
and used results relating the largest zero of Gegenbauer polynomials
to the first positive zero of a Bessel function, to show that such
point sets, which include spherical $t$-designs, have covering radius $\COV(Z_{N_t}) = \mathcal{O}_d( 1/t )$,where the order notation $\mathcal{O}_d$ means that the implied constant depends only on $d$.


Yudin's result implies that a sequence of spherical $t$-designs with $N_t \asymp t^d$ points has the optimal covering  property~\eqref{eq:optimal.covering.property}. Reimer's result also shows that the node sets of positive weight cubature rules that are exact for polynomials of degree at most $t$ and have $N = \mathcal{O}_d( t^d )$ points form a sequence that has the optimal covering property.
The present paper extends Yudin's result in a different direction, replacing the condition that polynomials of degree up to $t$ be integrated exactly by a condition on the rate of convergence of the QMC error. 

\subsection*{The results.} 
In this paper the worst-case error will play an important role. For a
Banach space $B$ of continuous functions on $\sph{d}$ with norm ${\| \cdot
\|_B}$, the {\em worst-case error} for the QMC method
$\numint[X_N]$ with node set $X_N \subset \sph{d}$  
approximating the integral $\xctint(f)$ is defined by
\begin{equation}\label{eq:wceB}
\WCE(\numint[X_N]; B ) \DEF \sup \left\{ \big| \numint[X_N](f) - \xctint(f) \big| : f \in B, \| f \|_B \leq 1 \right\}.
\end{equation}
That is, the worst-case error is the largest error (for the supremum is
indeed a maximum) for all functions in the unit ball of $B$.

We shall be interested in particular in the Sobolev spaces $\sob{p}{s}{d}$ for $p \geq 1$ and $s \geq 0$ consisting of functions $f \in \Lp{p}{d}$ for
which $(1-\LBo{d})^{s/2}f \in \Lp{p}{d}$, where $\LBo{d}$ is the Laplace-Beltrami operator on $\sph{d}$. The Sobolev norm $\| f \|_{\sob{p}{s}{d}}$ of  $f$ is defined to be the $\Lp{p}{d}$-norm $\| ( 1 - \LBo{d} )^{s/2} f \|_p$.
For a full description of the Sobolev space setting, see Section~\ref{sec:function.space.setting}.
We show in Section~\ref{sec:wce} that the worst-case error of $\numint[X_N]$ for $\sob{p}{s}{d}$ is equal to the $\Lp{q}{d}$-norm (with $1/p + 1/q = 1$) of a function that is related to the Bessel kernel for $\sob{p}{s}{d}$.

A principal result of the paper is that the covering radius of a point
set $X_N$ on $\sph{d}$ is upper bounded by a power of the worst-case error in a Sobolev space:

\begin{theorem}\label{thm:Covering.bounded.by.Wp}
Let $d \geq 1$, $1 \leq p, q \leq \infty$ with $1/p + 1/q = 1$ and $s > d/p$. For
a positive integer $N$, let $X_N$ be an $N$-point set on $\sph{d}$. Then
\begin{equation}\label{eq:covering.Wp}
\COV(X_N) \leq c_{s,d} \left[\WCE(\numint[X_N];\sob{p}{s}{d})\right]^{1 / ( s + d/q )},
\end{equation}
where the constant $c_{s,d}$ depends on $s$ and $d$ but not on $p$, $q$ or $N$.
\end{theorem}

The theorem will be proved in Section \ref{sec:covering}.  Note that the
condition $s > d/p$ is natural, in that it ensures that the generalized
Sobolev space is continuously embedded in the space of continuous
functions on $\sph{d}$.
The significance of \eqref{eq:covering.Wp} is that results on the order of
decay of worst-case cubature errors on $\sph{d}$ as $N\to \infty$
translate directly into bounds for the decay of the covering radius.  See Corollary \ref{cor:main} below for a concrete instance that ensures optimal order convergence of the covering radius.

The fact that spherical $t$-designs with $N_t \asymp t^d$ points have optimal order of
decay for the worst-case error in Sobolev spaces is a consequence of
results due to Brandolini et al.~\cite{BCCGST2013}, generalising
earlier results for $p=2$ of \cite{HeSl2005a} and \cite{BrHe2007}:

\begin{proposition}[{cf.~\cite[Lemma~2.10]{BCCGST2013}}]\label{prop:brandetal}
Let $1 \leq p \leq \infty$. Given $s > d/p$, there exists $C_{p,s,d} > 0$
such that for every $N$-point spherical $t$-design $X_N$ on $\sph{d}$ there holds
\begin{equation} \label{eq:asymptotic_in_t.bound}
\WCE(\numint[X_N];\sob{p}{s}{d}) \leq \frac{C_{p,s,d}}{t^s},
\end{equation}
where the constant $C_{p,s,d}$ does not depend on $t$, $N$, or the particular spherical design $X_N$.
\end{proposition}

Motivated by Propositions \ref{prop:Bondarenko.et.al.A} and \ref{prop:brandetal}, we extend the definition of QMC-design sequences as given in \cite{BrSaSlWo2014} from $p=2$ to general $p$:

\begin{definition} \label{def:specific_approx.sph.design}
Let $1 \leq p \leq \infty$. Given $s > d/p$, a sequence $(X_N)$ of $N$-point configurations on $\sph{d}$ with $N\to\infty$ is a \emph{QMC-design sequence for $\sob{p}{s}{d}$} if there exists $c_{p,s,d}>0$, independent of $N$, such that
\begin{equation} \label{eq:approxfors}
\WCE(\numint[X_N];\sob{p}{s}{d})  \leq \frac{c_{p,s,d}}{N^{s/d}}.
\end{equation}
\end{definition}

In this definition it is sufficient that $X_N$ exists for each $N$ in an infinite subset of the natural numbers.

The existence of spherical $t$-design sequences with $N_t \asymp t^d$ points (Proposition~\ref{prop:Bondarenko.et.al.A}) and Proposition~\ref{prop:brandetal} imply the existence of QMC-design sequences for $\sob{p}{s}{d}$:

\begin{theorem}[Existence of QMC-design sequences for $\sob{p}{s}{d}$] \label{thm:exist QMC for Wp}
For any $1 \leq p \leq \infty$ and $s > d/p$, there exists a QMC-design sequence for $\sob{p}{s}{d}$.
\end{theorem}
In particular, any sequence of minimizers of $\WCE(\numint[X_N]; \sob{p}{s}{d})$ for a fixed $s>d/p$ and an infinite number of values of $N$ is a QMC-design sequence for $\sob{p}{s}{d}$.

By a special case of \cite[Theorem~2.16]{BCCGST2013}, which
generalizes the earlier $p=2$ lower bounds of \cite{HeSl2005b} and
\cite{He2006}, the exponent of $N$ in~\eqref{eq:approxfors} cannot be
larger than~$s/d$:

\begin{proposition} \label{prop:generic.lower.bound}
Let $1 \leq p \leq \infty$. Given $s > d/p$, there exists $c_{p,s,d}^\prime > 0$ such that for any $N$-point configuration $X_N$ on $\sph{d}$,
\begin{equation} \label{eq:wce.lower.bound}
\WCE(\numint[X_N];\sob{p}{s}{d}) \geq \frac{c_{p,s,d}^\prime}{N^{s/d}}.
\end{equation}
\end{proposition}

\noindent Thus a QMC-design sequence for $\sob{p}{s}{d}$ yields error bounds of optimal order of convergence $N^{-s/d}$ for the worst-case error in $\sob{p}{s}{d}$ as $N \to \infty$.

As a consequence of Theorem~\ref{thm:Covering.bounded.by.Wp}, we obtain the following estimate for the covering radius for QMC-design sequences for $\sob{p}{s}{d}$, which is sharp when $p = 1$.

\begin{corollary} \label{cor:main}
Let $d \geq 1$, $1 \leq p,q \leq \infty$ with $1/p + 1/q = 1$. For a fixed $s > d/p$, let $( X_N )$ with $N \to \infty$ be a QMC-design sequence for $\sob{p}{s}{d}$. Then there exist a constant $c > 0$ depending on $d$, $s$, $p$ and the sequence $(X_N)$ but not on $N$ such that for all $X_N$
\begin{equation*}
\COV( X_N ) \leq c \, N^{-\beta/d}, \qquad \beta \DEF s / (s + d/q ).
\end{equation*}
In particular, if $p = 1$ (and thus $\beta = 1$ and $s > d$), then the sequence $(X_N)$ has the optimal covering property.
\end{corollary}

When $p=2$, an alternative approach to generating QMC-design sequences is to maximize the generalized sum of distances $\sum_{i=1}^N\sum_{j=i}^N|\bsx_i-\bsx_j|^{2s-d}$. Theorem~14 of \cite{BrSaSlWo2014} shows that such point sets minimize the worst-case error in $\mathbb{W}_2^s(\sph{d})$, $s\in (d/2, d/2+1)$, and thus form a QMC-design sequence for this Sobolev space.

\begin{example}
\label{ex:1}
Let $d\ge 1$, and for $\alpha\in (0,2)$ let $(X_N)$ be a sequence of $N$-point sets such that $X_N \subset \sph{d}$ is a maximizing
set for the generalized sum of distances,
\begin{equation*}
\sum_{i=1}^N \sum_{j=1}^N |\bsx_i-\bsx_j|^\alpha, \qquad \bsx_1,\bsx_2,\ldots,\bsx_N \in \sph{d}, 
\end{equation*}
where $|\cdot|$ denotes Euclidean distance in $\mathbb{R}^{d+1}$.  Then setting $p = 2$ in Corollary~\ref{cor:main},  
there exists $c>0$ such that
\begin{equation*}
\rho(X_N)\le c N^{-\beta/d} \qquad \text{with} \qquad \beta=(d+\alpha)/(2d+\alpha).
\end{equation*}
Note that for $d=2$ the rate approaches
$N^{-1/3}$ as $\alpha\to 2^-$ and $N^{-1/4}$ as $\alpha \to 0^+$. 
Further observe, that the bounds obtained here for any $d \geq 2$ are much better than those derived from using an area bound for the largest spherical cap that contains no points of $X_N$.
%
The estimates obtained in \cite{NaSuWa2010} for $\alpha \in (0,1)$ and \cite{Brauchart2008:optimal_logarithmic} as $\alpha \to 0^+$ yield that the coefficient $\beta$ above becomes $\beta^\prime = (d + \alpha) / [ d ( d + 2 ) ]$.
\end{example}


We also obtain the following lower bound on the covering radius:
\begin{theorem} \label{thm:Covering.lower.bound}
Let $d \geq 1$, $1 < p \leq \infty$ with $1/p + 1/q = 1$. For every fixed $s\in (d/p, d)$ there exists a QMC-design sequence $( X_N )$ for $\sob{p}{s}{d}$ such that
\begin{equation}
\COV(X_N) \geq c_{p,s,d}^\prime \; N^{- s/d^2} \qquad \text{for all $X_N$,}
\end{equation}
where the constant $c_{s,d}^\prime$ depends on $p$, $s$ and $d$ but not on $N$.
\end{theorem}

For values of $p$  larger than $1$ and $s$ a fixed number in $(d/p,d)$, Theorem~\ref{thm:Covering.lower.bound} shows that there exists a QMC-design sequence $( X_N )$ for $\sob{p}{s}{d}$ that does \emph{not} have the optimal covering property \eqref{eq:optimal.covering.property} because $s/d^2<1/d$.  

The next two theorems assert conditions under which a QMC-design sequence for $\sob{p}{s}{d}$ retains the QMC-design property if the parameters $p$ and $s$ are changed. These results are proved in Section~\ref{subsec:WCE.inequality} using lemmas from Sections~\ref{sec:auxiliary.results} and~\ref{sec:Bernstein.inequality}.

\begin{theorem} \label{thm:QMC.desing.property.A}
Let $d \geq 1$, $1 \leq p < \infty$ and $s > d / p$. A QMC-design sequence $( X_N )$ for $\sob{p}{s}{d}$ is also a QMC-design sequence for $\sob{p^\prime}{s}{d}$ for all $p^\prime$ satisfying $p < p^\prime \leq \infty$.
\end{theorem}

The second theorem makes use of the \textit{mesh ratio} 
\begin{equation} \label{eq:mesh.ratio}
\MESHRATIO( X_N ) \DEF \COV( X_N ) / \SEP( X_N ),
\end{equation}
of an $N$-point configuration $X_N = \{ \PT{x}_1, \ldots, \PT{x}_N \} \subset \sph{d}$, where the \emph{separation distance} of~$X_N$ is given by
\begin{equation} \label{eq:separation.distance}
\SEP( X_N ) \DEF \min_{\substack{1 \leq j, k \leq N \\ j \neq k}} \arccos( \PT{x}_j \cdot \PT{x}_k ).
\end{equation} 
A sequence $( X_N )$ of $N$-point sets on $\sph{d}$ is \emph{well-separated} if there is a positive constant $c$ such that $\SEP(X_N)\ge c \, N^{-1/d}$ and \textit{quasi-uniform} provided $\MESHRATIO( X_N )$ is uniformly bounded in~$N$.

\begin{theorem} \label{thm:QMC.desing.property.B} Let $d \geq 1$, $1\leq p, q \leq\infty$ satisfying $1/p + 1/q = 1$ and $s > s^\prime > d/p$.
Then there exists a constant $c>0$, depending on $p, s^\prime, s$, and $d$ but independent of $N$, such that for every $N$-point node set $X_N \subset \sph{d}$,
\begin{equation}\label{newthmquasi}
\WCE(\numint[X_N]; \sob{p}{s^\prime}{d} ) \leq c \, \left[ \MESHRATIO( X_N ) \right]^{d/p} N^{(s - s^\prime) / d} \, \WCE(\numint[X_N]; \sob{p}{s}{d} ).
\end{equation}

Consequently, a quasi-uniform QMC-design sequence~$(X_N)$ for $\sob{p}{s}{d}$ is also a QMC-design sequence for $\sob{p}{s^\prime}{d}$ for all $s^\prime$ satisfying $s > s^\prime > d/p$.
\end{theorem}
A substantial part of the paper is devoted to establishing the estimate (\ref{newthmquasi}).
(Although needed for our argument, it is plausible that the quasi-uniformity assumption in the second assertion of Theorem \ref{thm:QMC.desing.property.B} can be removed.)

The structure of the paper is as follows. In the next section we prove Theorems \ref{thm:Covering.bounded.by.Wp} and~\ref{thm:Covering.lower.bound} and we extend Theorem~\ref{thm:Covering.bounded.by.Wp} to take into account the radii of several caps excluding points of~$X_N$.
In Section~\ref{sec:function.space.setting}, we discuss the function space setting and introduce the Bessel kernel for $\sob{p}{s}{d}$. In Section~\ref{sec:wce}, we present a worst-case error formula in terms of a Bessel kernel, which is used to prove embedding type results for QMC-design sequences for $\sob{p}{s}{d}$ when $p$ and $s$ vary. 
In Section~\ref{sec:auxiliary.results}, we introduce a special filtered kernel that enables us to prove a boundedness result for the Bessel kernel. In Section~\ref{sec:Bernstein.inequality}, such filtered kernels are further used to prove a Bernstein type inequality for the Bessel kernel which is needed for the proof of the inequality~\eqref{newthmquasi}.
Section~\ref{sec:unit.circle} considers the special case of the unit circle (i.e., the sphere $\sph{1}$). 

\pagebreak

\section{Bounds for the Covering Radius}
\label{sec:covering}

In this section we give the proofs of Theorems \ref{thm:Covering.bounded.by.Wp}, and \ref{thm:Covering.lower.bound}.

\subsection{Upper bound}

For the proof of Theorem~\ref{thm:Covering.bounded.by.Wp} we shall make use of the following interpolation inequality (of Gagliardo-Nirenberg type) on the sphere (see \cite{BeLo1976}).

\begin{lemma} \label{lem:Gagliardo-Nirenberg.inequality}
Let $d\geq1$, $1 \leq p \leq \infty$ and $0 \leq s_{0} < s < s_{1} < \infty$. Then there exists a constant~$c$ depending only on $s$, $p$, and $d$ such that for any $0 \leq \theta \leq 1$ and $s=(1-\theta)s_{0}+\theta s_{1}$, we have
\begin{equation}\label{eq:interp.W.s}
\left\| f \right\|_{\sob{p}{s}{d}} \leq c \left\| f \right\|_{\sob{p}{s_0}{d}}^{1-\theta} \left\| f \right\|_{\sob{p}{s_1}{d}}^{\theta}.
\end{equation}
\end{lemma}

\begin{proof}[Proof of Theorem~\ref{thm:Covering.bounded.by.Wp}]
For a node set $X_N = \{ \PT{x}_1, \ldots, \PT{x}_N \}$ on $\sph{d}$, we construct a ``fooling function'' made up of a bump that is supported on a ``spherical collar'' contained in the largest hole of $X_N$. The outer radius of this collar is chosen to be $\rho = \COV( X_N )$. The function~$f_{\rho}$ is defined so that it is zero at every point of $X_N$, thus providing a lower bound for the worst-case error (cf.~\eqref{eq:wceB}):
\begin{equation} \label{eq:WCE.lowerbound}
\WCE( \numint[X_N]; \sob{p}{s}{d} ) \geq \frac{|\xctint( f_\rho ) |}{\|f_\rho\|_{\sob{p}{s}{d}}}. 
\end{equation}
We shall show that the right-hand side  can be lower bounded in terms of $\rho$.

For the precise definition of $f_{\rho}$, we appeal to results of Hesse~\cite{He2006} by starting with the symmetric $C^\infty( \R )$ function with support $[-1,1]$,
\begin{equation} \label{eq:bump}
\Phi(t) \DEF
\begin{cases}
\exp\Big( 1 - \dfrac{1}{1-t^2} \Big) & \text{if $-1 < t < 1$,} \\
0 & \text{otherwise.}
\end{cases}
\end{equation}
We rescale this function to have new support $[\cos \rho, \cos( \rho / 2 ) ]$ using the linear bijection $g_\rho$ that maps this interval onto $[-1,1]$ giving $\Phi_\rho( t ) \DEF \Phi( g_\rho( t ) )$ and then lift it to the sphere to get the zonal function
\begin{equation} \label{eq:zonal.duping.function}
f_\rho( \PT{x} ) \DEF \Phi_\rho( \PT{y}_0 \cdot \PT{x} ) = \Phi( g_\rho( \PT{y}_0 \cdot \PT{x} ) ), \qquad \PT{x}, \PT{y}_0 \in \sph{d}.
\end{equation}
The point $\PT{y}_0$ is chosen to be the center of a largest hole and thus achieves the maximum in \eqref{eq:covering_def}.
It is easily seen that $f_\rho \in C^\infty(\sph{d})$ and that $f_\rho(\PT{x})$ vanishes unless $\cos(\rho/2) \geq \PT{y}_0 \cdot \PT{x} \geq \cos \rho$, so that the support of $f_\rho$ is a collar within the spherical cap
\begin{equation*}
S(\PT{y}_0; \rho) \DEF \{\PT{x} \in \sph{d}: \PT{y}_0 \cdot \PT{x} \ge \cos \rho\}.
\end{equation*}

We now estimate the quantities on the right-hand side of \eqref{eq:WCE.lowerbound}.
For $\xctint( f_\rho )$, the Funk-Hecke formula and a change of variable gives (also cf.~\cite[Eq.~(32)]{He2006})
\begin{equation} \label{eq:int.f.alpha}
\left| \xctint( f_\rho ) \right| = \int_{\sph{d}} f_\rho( \PT{x} ) \, \dd \sigma_d(\PT{x}) = \frac{\omega_{d-1}}{\omega_d} \int_{\rho/2}^\rho \Phi( g_\rho( \cos \theta ) ) \left( \sin \theta \right)^{d-1} \dd \theta \geq c_d \, \rho^d,
\end{equation}
where $\omega_d$ is the surface area of $\sph{d}$.
The Sobolev norm of $f_{\rho}$, computed as $\| f_\rho \|_{\sob{p}{s}{d}} = \| ( 1 - \LBo{d} )^{s/2} f_\rho  \|_{p}$, is first estimated for even~$s$. The result for other $s$ is then obtained from the even case using Lemma~\ref{lem:Gagliardo-Nirenberg.inequality}.
Let $s$ be a non-negative even integer. Since $f_\rho$ is zonal, use of spherical cylinder coordinates (cf.~\cite{Mu1966}) gives for $d \geq 1$:
\begin{equation} \label{eq:LBo.to.f.rho}
\left( 1 - \LBo{d} \right)^{s/2} f_\rho ( \PT{x} ) = \left( 1 + d \, t \, \mathrm{D}_t - \left( 1 - t^2 \right) \mathrm{D}_{t}^2 \right)^{s/2} \Phi_\rho( t ), \quad t = \PT{y}_0 \cdot \PT{x},
\end{equation}
where $\mathrm{D}_t \DEF \dd / \dd t$.
Expansion of the differential operator and term-wise estimation gives
\begin{equation*}
\left| \left( 1 - \LBo{d} \right)^{s/2} f_\rho( \PT{x} ) \right| \leq c_{s,d} \, \left( \cos \frac{\rho}{2} - \cos \rho \right)^{-s/2} \leq c_{s,d}^\prime \, \rho^{-s}, \qquad \PT{x} \in \sph{d}.
\end{equation*}
The details involve a slight modification of the arguments in \cite{He2006}, where a different constant is used in the differential operator $( 1 - \LBo{d} )^{s/2}$. The computations can also then be extended to include $d = 1$, a case not considered in \cite{He2006}.
Thus for $p = \infty$, 
\begin{equation*}
\left\| f_\rho \right\|_{\sob{\infty}{s}{d}} = \sup_{ \PT{x} \in \sph{d} } \left| \left( 1 - \LBo{d} \right)^{s/2} f_\rho( \PT{x} ) \right| \leq c_{s,d}^\prime \, \rho^{-s},
\end{equation*}
while for $1 \leq p < \infty$, since $f_\rho$ is supported in $S(\PT{y}_0; \rho)$,
\begin{equation*}
\left\| f_\rho \right\|_{\sob{p}{s}{d}}^p = \int_{\sph{d}} \big| \left( 1 - \LBo{d} \right)^{s/2} f_\rho( \PT{x} ) \big|^p \dd\sigma_{d}(\PT{x}) \leq ( c_{s,d}^\prime )^p \, \rho^{- p \, s} \, \sigma_d( S(\PT{y}_0; \rho) ) \leq ( c_{s,d}^{\prime\prime} )^p \, \rho^{- p \, s + d}.
\end{equation*}
Hence
\begin{equation} \label{eq:Sobolev.norm.lower.bound.even.case}
\left\| f_\rho \right\|_{\sob{p}{s}{d}} \leq c_{s,d}^{\prime\prime\prime} \, \rho^{-s + d/p}, \qquad \text{$1 \leq p \leq \infty$, $s \geq 0$ even.}
\end{equation}

For general $s$, write $s = 2L + 2\theta$ with $0 \leq \theta \leq 1$ and $L$ a non-negative integer. Then we can apply the interpolation inequality~\eqref{eq:interp.W.s} with $s_0 = 2L$ and $s_1 = 2L + 2$ to obtain again
\begin{equation} \label{eq:Sobolev.norm.lower.bound.other.case}
\left\| f_\rho \right\|_{\sob{p}{s}{d}} \leq c_{s,d}^{iv} \, \rho^{-s + d/p}, \qquad \text{$1 \leq p \leq \infty$, $s \geq 0$.}
\end{equation}
Finally, using the estimates \eqref{eq:int.f.alpha} and \eqref{eq:Sobolev.norm.lower.bound.other.case} in \eqref{eq:WCE.lowerbound}, we get
\begin{equation*}
\WCE( \numint[X_N]; \sob{p}{s}{d} ) \geq c_{s,d}^{v} \, \rho^{s + d/q},
\end{equation*}
where $1\leq q \leq \infty$ is such that $1/p + 1/q = 1$. The proof is now complete.
\end{proof}

\subsection{A generalization of the upper bound} 
\label{appdx:generalization.covering.result}

We now provide a generalization of Theorem~\ref{thm:Covering.bounded.by.Wp}.  Another way of describing the covering radius~$\COV(X_N)$ is as the largest hole radius -- more precisely, as the geodesic radius of the largest open spherical cap on $\sph{d}$ that does not contain a point of $X_N$. A (spherical cap shaped) hole of $X_N$ can be any open spherical cap on $\sph{d}$ that does not contain points of~$X_N$. Of particular interest are maximal holes, which are ones that lie ``above'' the facets of the convex hull of~$X_N$. They are said to be maximal as they cannot be enlarged. Indeed, the supporting plane of a facet divides the sphere into an open spherical cap that contains no points of $X_N$ and a closed one that contains all the points of~$X_N$. These maximal holes provide a natural covering of the sphere with, in general, differently sized spherical caps of maximal radii.
Among these maximal holes one can select a sequence of pairwise disjoint holes ordered with respect to non-increasing radii. This is a particular example of what we will call an ``ordered $X_N$-avoiding packing on~$\sph{d}$.''

\begin{definition} \label{def:ordered.points.avoiding.packing}
Given an $N$-point set $X_N = \{ \PT{x}_1, \ldots, \PT{x}_N \}$ on $\sph{d}$, a sequence $( s_n )_{n \geq 1}$ of open pairwise disjoint spherical caps with cap radii $\rho_1 \geq \rho_2 \geq \rho_3 \geq \cdots$ such that no $s_n$ contains points of $X_N$ is called an \emph{ordered $X_N$-avoiding packing on $\sph{d}$}.
\end{definition}

The next theorem gives an upper bound of the $n$th-largest spherical cap radius in an ordered $X_N$-avoiding packing of $\sph{d}$.

\begin{theorem}
Let $d\geq1$, $1\leq p,q\leq\infty$ such that $1/p + 1/q = 1$ and $s > d/p$. Given an $N$-point set $X_N$ on $\sph{d}$ and an ordered $X_N$-avoiding packing on $\sph{d}$ with spherical cap radii $\rho_1 \geq \rho_2 \geq \rho_3 \geq \cdots$, then
\begin{equation*}
\rho_n \leq c_{s,d} \, n^{-1 / ( q s + d )} \, \left[ \WCE( \numint[X_N]; \sob{p}{s}{d} ) \right]^{1 / ( s + d/q )},
\end{equation*}
where the constant $c_{s,d}$ depends on $s$ and $d$ but not on $p$ or $q$ or the packing.
\end{theorem}

\begin{proof} We proceed along the same lines as the proof of Theorem~\ref{thm:Covering.bounded.by.Wp} but use now a fooling function of the form
\begin{equation*}
F_n( \PT{x} ) \DEF \sum_{k = 1}^n f_{\rho_k}( \PT{x} ), \qquad \PT{x} \in \sph{d},
\end{equation*}
that is the sum of all the contributions of functions $f_{\rho_k}( \PT{x} ) \DEF \Phi_{\rho_k}( \PT{y}_k \cdot \PT{x} )$ of type~\eqref{eq:zonal.duping.function} fitted to holes centered at $\PT{y}_k \in \sph{d}$, $1 \leq k \leq n$, in the ordered $X_N$-avoiding packing on $\sph{d}$.

From \eqref{eq:int.f.alpha} we obtain that
\begin{equation} \label{eq:ext.int.duping.function.F.n}
\left| \xctint( F_n ) \right| = \sum_{k=1}^n \left| \xctint( f_{\rho_k} ) \right| \geq c_d \sum_{k=1}^n \rho_k^d \geq n \, c_d \, \rho_n^d.
\end{equation}
Observe that the supports of any two of $f_{\rho_1}, \ldots, f_{\rho_n}$ intersect at most on their boundaries. It follows that for even $s \geq 0$ this property also holds for any two of $\BeOp{-s}[f_{\rho_1}], \ldots, \BeOp{-s}[f_{\rho_n}]$, where $\BeOp{-s}[f_\rho] = \left( 1 - \LBo{d} \right)^{s/2} f_\rho$ (see \eqref{eq:LBo.to.f.rho}). 
The estimate \eqref{eq:Sobolev.norm.lower.bound.even.case} gives for $p = \infty$, 
\begin{equation*}
\left\| F_n \right\|_{\sob{p}{s}{d}} = \max_{1 \leq m \leq n} \left\| f_{\rho_m} \right\|_{\sob{p}{s}{d}} \leq \max_{1 \leq m \leq n} c_{s,d}^{\prime\prime\prime} \, \rho_m^{-s} = c_{s,d}^{\prime\prime\prime} \, \rho_n^{-s},
\end{equation*}
while for $1 \leq p < \infty$,
\begin{equation*}
\left\| F_n \right\|_{\sob{p}{s}{d}}^p = \sum_{m=1}^n \left\| f_{\rho_m} \right\|_{\sob{p}{s}{d}}^p  \leq \sum_{m=1}^n \left( c_{s,d}^{\prime\prime\prime} \, \rho_m^{-s + d/p} \right)^p \leq n \left( c_{s,d}^{\prime\prime\prime} \, \rho_n^{-s + d/p} \right)^p.
\end{equation*}
Hence
\begin{equation} \label{eq:Sobolev.norm.F.n.lower.bound.even.case}
\left\| F_n \right\|_{\sob{p}{s}{d}} \leq n^{1/p} \, c_{s,d}^{\prime\prime\prime} \, \rho^{-s + d/p}, \qquad \text{$1 \leq p \leq \infty$, $s \geq 0$ even.}
\end{equation}
The result for other $s$ is then obtained using the interpolation inequality \eqref{eq:interp.W.s}:
\begin{equation} \label{eq:Sobolev.norm.F.n.lower.bound.other.case}
\left\| F_n \right\|_{\sob{p}{s}{d}} \leq n^{1/p} \, c_{s,d}^{iv} \, \rho^{-s + d/p}, \qquad \text{$1 \leq p \leq \infty$, $s \geq 0$ not even.}
\end{equation}

Substituting the estimates \eqref{eq:ext.int.duping.function.F.n}, \eqref{eq:Sobolev.norm.F.n.lower.bound.even.case} and \eqref{eq:Sobolev.norm.F.n.lower.bound.other.case} into \eqref{eq:WCE.lowerbound}, we get
\begin{equation*}
\WCE( \numint[X_N]; \sob{p}{s}{d} ) \geq n^{1/q} \, c_{s,d}^{v} \, \rho_n^{s + d/q},
\end{equation*}
where $1\leq q \leq \infty$ is such that $1/p + 1/q = 1$. This completes the proof.
\end{proof}

\subsection{Lower bound}

\begin{proof}[Proof of Theorem~\ref{thm:Covering.lower.bound}]
Let $( Z_{N_t} )$ be a sequence of well-separated spherical $t$-designs on $\sph{d}$ with $N_t \asymp t^d$. The existence of such a sequence is established in \cite{BoRaVi2014}. Fix $\varepsilon \in (0,1]$ and $c > 0$. For each $Z_{N_t}$ we select a spherical cap with radius $\alpha_t = c \, N_t^{-(1 - \varepsilon) / d}$ and an arbitrary center, and remove all the points in this cap. This gives a new set $X_{N_t-M_t}$ with $N_t - M_t$ points, where $M_t$ depends on the cap and on $\alpha_t$ and thus on $N_t$. It follows from the well-separation property of $(Z_{N_t})$ that $\SEP( X_{N_t-M_t} ) \geq c^\prime N_t^{-1/d}$ for some $c^\prime > 0$, thus for some $c^{\prime\prime} > 0$, we have
\begin{equation*}
M_t \leq c^{\prime\prime} \, N_t^{\varepsilon} \qquad \text{for all $Z_{N_t}$.}
\end{equation*}
The removal of the $M_t$ points generates a hole of radius $\alpha_t$, so that the covering radius of $X_{N_t-M_t}$ satisfies
\begin{equation} \label{eq:lower.bound.cov}
\COV( X_{N_t-M_t} ) \geq \alpha_t = c \, N_t^{- ( 1 - \varepsilon ) / d}.
\end{equation}
Next, we quantify the quality of $X_{N_t-M_t}$ as a set of cubature points by estimating the worst-case error for the QMC method $\numint[X_{N_t-M_t}]$. Let $f \in \sob{p}{s}{d}$, $s > d/p$, with ${\| f \|_{\sob{p}{s}{d}} = 1}$. The error of  numerical integration, $\mathcal{R}[X_{N_t-M_t}]( f ) \DEF \numint[X_{N_t-M_t}]( f ) - \xctint[X_{N_t-M_t}]( f )$, can be written as
\begin{equation*}
\mathcal{R}[X_{N_t-M_t}]( f ) = \frac{N_t}{N_t-M_t} \, \mathcal{R}[Z_{N_t}]( f ) - \frac{M_t}{N_t-M_t} \, \mathcal{R}[ Z_{N_t} \setminus X_{N_t-M_t}]( f ).
\end{equation*}
Since $(Z_{N_t})$ is a sequence of spherical $t$-designs with $N_t \asymp t^d$ and $N_t / ( N_t - M ) \to 1$ as $N_t \to \infty$, it follows from Proposition~\ref{prop:brandetal} that for some $C > 0$ we have
\begin{equation*}
\frac{N_t}{N_t-M_t} \left| \mathcal{R}[Z_{N_t}]( f ) \right| \leq C \, N_t^{-s/d}.
\end{equation*}
Furthermore, the fact that $\sob{p}{s}{d}$ can be continuously embedded into $C( \sph{d} )$ for $s > d/p$ (Proposition~\ref{prop:embed.Wp}) gives that for some $c_{p,s,d} > 0$ (embedding constant)
\begin{equation*}
\left| \mathcal{R}[ Z_{N_t} \setminus X_{N_t-M_t}]( f ) \right| \leq 2 \, \sup_{\PT{x} \in \sph{d}} | f( \PT{x} ) | \leq 2 \, c_{p,s,d} \left\| f \right\|_{\sob{p}{s}{d}} = 2 \, c_{p,s,d}.
\end{equation*}
Since $M_t / ( N_t - M_t ) = \mathcal{O}( N_t^{-( 1 - \varepsilon )} )$, we get
\begin{equation*}
\frac{M_t}{N_t-M_t} \left| \mathcal{R}[ Z_{N_t} \setminus X_{N_t-M_t}]( f ) \right| \leq 2 \, c_{p,s,d} \, \frac{M_t}{N_t-M_t} \leq C^\prime \, N_t^{-( 1 - \varepsilon )}.
\end{equation*}
We conclude that
\begin{equation}
\WCE(\numint[X_{N_t-M_t}]; \sob{p}{s}{d} ) \leq C \, N_t^{-s/d} + C^\prime \, N_t^{-( 1 - \varepsilon )}.
\end{equation}
Until now we have allowed $\varepsilon \in (0,1]$ to be arbitrary. If we now force $\varepsilon \DEF 1 - s/d$, then $( X_{N_t-M_t} )$ is a well-separated QMC-design sequence for $\sob{p}{s}{d}$. From \eqref{eq:lower.bound.cov} we have $\COV( X_{N_t-M_t} ) \geq c \, N_t^{-s/d^2}$, completing the proof.
\end{proof}

A more precise analysis of the effects on the worst-case error when one or more points are removed from a circular design (i.e., a spherical design on $\sph{1}$) is given in Section~\ref{sec:unit.circle}.

\section{The function space setting and embedding theorems}
\label{sec:function.space.setting}

In this section we set up the machinery needed to prove the worst-case error results.
Let~$d$ be a positive integer. Our manifold is the unit sphere $\sph{d}$ in the Euclidean space $\R^{d+1}$ provided with the normalized surface area measure $\sigma_d$. For future reference we record that
\begin{equation} \label{eq:omega.d.ratio}
\frac{\omega_{d-1}}{\omega_d} = \frac{\gammafcn( (d + 1)/2 )}{\sqrt{\pi} \, \gammafcn( d/2 )}, \qquad 2^{d-1} \frac{\omega_{d-1}}{\omega_d} \int_{-1}^1 \left( 1 - t^2 \right)^{d/2-1} \dd t = 1,
\end{equation}
where $\gammafcn( z )$ denotes the gamma function, and $\omega_d$ is the surface area of $\sph{d}$.

\subsection{Spherical harmonics}
The restriction to $\sph{d}$ of a homogeneous and harmonic polynomial of total degree $\ell$ defined on $\R^{d+1}$ is called a \emph{spherical harmonic} of degree $\ell$ on~$\sph{d}$. The family $\mathcal{H}_{\ell}^{d} = \mathcal{H}_{\ell}^{d}( \sph{d} )$ of all spherical harmonics of exact degree $\ell$ on $\sph{d}$ has dimension
\begin{equation*}
Z(d,\ell) \DEF \left( 2\ell + d - 1 \right) \frac{\gammafcn( \ell + d - 1 )}{\gammafcn( d ) \gammafcn( \ell + 1 )}.
\end{equation*}
Each spherical harmonic $Y_{\ell}$ of exact degree $\ell$ is an eigenfunction of the negative \emph{Laplace-Beltrami operator} $-\LBo{d}$ for $\sph{d}$, with eigenvalue
\begin{equation} \label{eq:eigenvalue}
\lambda_\ell \DEF \ell \left( \ell + d - 1 \right), \qquad \ell = 0, 1, 2, \ldots.
\end{equation}
As usual, let $\{ Y_{\ell, k} : k = 1, \ldots, Z(d, \ell) \}$ denote an $\mathbb{L}_2$-orthonormal basis of $\mathcal{H}_{\ell}^{d}$. Then the basis functions $Y_{\ell, k}$ satisfy the following identity known as the \emph{addition theorem}:
\begin{equation} \label{eq:addition.theorem}
\sum_{k=1}^{Z(d,\ell)} Y_{\ell,k}( \PT{x} ) Y_{\ell,k}( \PT{y} ) = Z(d,\ell) \, P_\ell^{(d)}(\PT{x} \cdot \PT{y}), \qquad \PT{x}, \PT{y} \in \sph{d},
\end{equation}
where $P_\ell^{(d)}$ is the normalized Gegenbauer (or Legendre) polynomial, orthogonal on the interval $[-1,1]$ with respect to the weight function $(1-t^2)^{d/2-1}$, and normalized by $P_\ell^{(d)}(1) = 1$.

The collection $\{ Y_{\ell, k} : k = 1, \ldots, Z(d, \ell); \ell = 0, 1, \ldots \}$ forms a complete orthonormal (with respect to $\sigma_d$) system for the Hilbert space $\mathbb{L}_2(\sph{d})$ of square-integrable functions on~$\sph{d}$ endowed with the usual inner product 
\begin{equation*}
( f, g )_{\Lp{2}{d}} \DEF \int_{\sph{d}} f( \PT{x} ) g( \PT{x} ) \dd \sigma_d( \PT{x} ), 
\end{equation*}
as well as a complete system for all the Banach spaces $\Lp{p}{d}$ of $p$th power integrable functions on $\sph{d}$ with $1 \leq p < \infty$ provided with the usual $p$-norm
\begin{equation*}
\left\| f \right\|_p \DEF \left\| f \right\|_{\Lp{p}{d}} \DEF \left( \int_{\sph{d}} \left| f( \PT{x} ) \right|^p \dd \sigma_d( \PT{x} ) \right)^{1/p},
\end{equation*}
and for the Banach space $C( \sph{d} )$ of continuous functions on $\sph{d}$ endowed with the maximum norm
\begin{equation*}
\left\| f \right\|_{C} \DEF \max_{\PT{x} \in \sph{d}} \left| f( \PT{x} ) \right|.
\end{equation*}
(For more details, we refer the reader to \cite{BeBuPa1968,Mu1966}.)

The Funk-Hecke formula states that for every spherical harmonic $Y_\ell$ of degree $\ell$ (see~\cite{Mu1966}),
\begin{equation}
\label{eq:Funk-Hecke.formula}
\int_{\sph{d}} g( \PT{y} \cdot \PT{z} ) \, Y_\ell( \PT{y} ) \, \dd \sigma_d( \PT{y} ) = \widehat{g}(\ell) \, Y_\ell( \PT{z} ), \qquad \PT{z} \in \sph{d},
\end{equation}
where
\begin{equation} \label{eq:Funk-Hecke.formula.B}
\widehat{g}( \ell ) = \frac{\omega_{d-1}}{\omega_d} \, \int_{-1}^1 g( t ) \, P_\ell^{(d)}( t ) \left( 1 - t^2 \right)^{d/2-1} \dd t.
\end{equation}
(This formula holds, in particular, for the spherical harmonic ${Y_\ell( \PT{y} ) = P_\ell^{(d)}( \PT{a} \cdot \PT{y} )}$, $\PT{a} \in \sph{d}$.)

\subsection{Convolution}
We shall frequently use the convolution of a zonal kernel, i.e. one that depends only on the inner product of the arguments, ``against'' a function $f$ on $\sph{d}$.
With abuse of notation we write $G( \PT{x}, \PT{y} ) = G( \PT{x} \cdot \PT{y} )$ for $\PT{x}, \PT{y} \in \sph{d}$.
For $1\leq p<\infty$, let $\mathbb{L}_{p,d}([-1,1])$ consists of all functions of the form $g_{\PT{z}}( \PT{x} ) \DEF G( \PT{z} \cdot \PT{x} )$, $\PT{x}, \PT{z} \in \sph{d}$, with finite norm $\| G \|_{p,d} \DEF \| g_{\PT{z}} \|_{p}$. Of course, this norm does not depend on the choice of $\PT{z} \in \sph{d}$ since, by the Funk-Hecke formula with $\ell = 0$ (see~\eqref{eq:Funk-Hecke.formula} and \eqref{eq:Funk-Hecke.formula.B}),
\begin{equation}
\left\| G \right\|_{p,d} = \| g_{\PT{z}} \|_p = \left( \frac{\omega_{d-1}}{\omega_d} \int_{-1}^1 \left| G( t ) \right|^p \left( 1 - t^2 \right)^{d/2-1} \dd t \right)^{1/p}.
\end{equation}

\begin{definition}
The convolution of the zonal kernel $G \in \mathbb{L}_{1,d}([-1,1])$ against $f \in \Lp{p}{d}$ is the function $G \ast f$ given by
\begin{equation*}
( G \ast f )( \PT{x} ) \DEF \int_{\sph{d}} G( \PT{z} \cdot \PT{x} ) \, f( \PT{z} ) \, \dd \sigma_d( \PT{z} ), \qquad \PT{x} \in \sph{d}.
\end{equation*}
The convolution of the zonal kernel $K \in \mathbb{L}_{1,d}([-1,1])$ against $G \in \mathbb{L}_{1,d}([-1,1])$ is the kernel $K \ast G$ given by
\begin{equation*}
( K \ast G )( \PT{x} \cdot \PT{y} ) \DEF \int_{\sph{d}} K( \PT{z} \cdot \PT{x} ) \, G( \PT{z} \cdot \PT{y} ) \, \dd \sigma_d( \PT{z} ), \qquad \PT{x}, \PT{y} \in \sph{d}.
\end{equation*}
\end{definition}

If $g \in \mathbb{L}_{q,d}([-1,1])$, $1 \leq p, q \leq \infty$ and $f \in \Lp{p}{d}$, then the convolution $g \ast f$ exists $\sigma_d$-almost everywhere on $\sph{d}$ and Young's inequality holds; i.e.,
\begin{equation} \label{eq:Young.s.ineq}
\left\| g \ast f \right\|_{r} \leq \left\| g \right\|_{q,d} \, \left\| f \right\|_{p}  \qquad \text{for all $r$ with $\frac{1}{r} = \frac{1}{p} + \frac{1}{q} - 1 \geq 0$.}
\end{equation}
In particular, one has
\begin{equation} \label{eq:Young.s.ineq.A}
\left\| g \ast f \right\|_{p} \leq \left\| g \right\|_{1,d} \, \left\| f \right\|_{p} \qquad \text{and} \qquad \left\| g \ast f \right\|_{q} \leq \left\| g \right\|_{q,d} \, \left\| f \right\|_{1}.
\end{equation}

\subsection{Sobolev function space classes}
\label{subsec:generalized.Sobolev.space}

The \emph{Laplace-Fourier series} (in terms of spherical harmonics) of a function $f \in \Lp{1}{d}$ is given by the formal expansion
\begin{equation} \label{eq:spherical.harmonics.expansion}
S[f]( \PT{x} ) \sim \sum_{\ell = 0}^\infty Y_\ell[f]( \PT{x} ), \qquad \PT{x} \in \sph{d},
\end{equation}
where $Y_\ell[f]$ is the \emph{projection of $f$ onto $\mathcal{H}_\ell^d$}. It can be obtained by the convolution
\begin{equation} \label{eq:projection.onto.calligr.H.A}
Y_\ell[f]( \PT{x} ) \DEF \int_{\sph{d}} Z(d, \ell) \, P_\ell^{(d)}(\PT{x} \cdot \PT{y}) \, f( \PT{y} ) \, \dd \sigma_d( \PT{y} ), \qquad \PT{x} \in \sph{d}.
\end{equation}
Application of the addition theorem yields
\begin{equation} \label{eq:projection.onto.calligr.H.B}
Y_\ell[f]( \PT{x} ) = \sum_{k = 1}^{Z(d, \ell)} \widehat{f}_{\ell, k} \, Y_{\ell, k}( \PT{x} ), \qquad \PT{x} \in \sph{d},
\end{equation}
and $\widehat{f}_{\ell, k}$ are the \emph{Laplace-Fourier coefficients of $f$} defined by
\begin{equation} \label{eq:L-F.coeff}
\widehat{f}_{\ell, k} \DEF \int_{\sph{d}} f( \PT{x} ) \, Y_{\ell,k}( \PT{x} ) \dd \sigma_{d}(\PT{x}), \qquad k = 1, \ldots, Z(d, \ell), \, \ell = 0, 1, 2, \ldots.
\end{equation}

\begin{definition} \label{def:generalized.Sobolev.space}
The generalized Sobolev space $\sob{p}{s}{d}$ may be defined for $s \geq 0$ and $1 \leq p \leq \infty$ as the set of all functions $f\in \Lp{p}{d}$ with
\begin{equation} \label{eq:primary.Sobolev.space.characterization}
\left\| f \right\|_{\sob{p}{s}{d}} \DEF \left\| \sum_{\ell=0}^\infty \left( 1 + \lambda_\ell \right)^{s/2} Y_\ell[f] \right\|_p < \infty,
\end{equation}
where the $\lambda_\ell$ are given in \eqref{eq:eigenvalue} and formulas for $Y_\ell[f]$ are provided in \eqref{eq:projection.onto.calligr.H.A} and \eqref{eq:projection.onto.calligr.H.B}.
\end{definition}

\begin{remark*}
The definition implies that $\sum_{\ell=0}^L \left( 1 + \lambda_\ell \right)^{s/2} Y_\ell[f]( \PT{x} )$ converges pointwise as ${L \to \infty}$ for almost all (in the sense of Lebesgue measure) points on $\sph{d}$, since otherwise the sum is not in $\Lp{p}{d}$.
\end{remark*}

We note that $\sob{p}{0}{d} = \Lp{p}{d}$. In the case of $p = 2$, Parseval's identity yields the following equivalent characterization: a function $f \in \Lp{2}{d}$ is in $\sob{2}{s}{d}$ if and only if the Laplace-Fourier coefficients $\widehat{f}_{\ell, k}$ of $f$ given in \eqref{eq:L-F.coeff} satisfy the condition
\begin{equation} \label{eq:sobcond}
\sum_{\ell=0}^\infty \left( 1 + \lambda_\ell \right)^{s} \sum_{k=1}^{Z(d,\ell)} \left| \widehat{f}_{\ell,k} \right|^2 < \infty,
\end{equation}
but a characterization of this kind in terms of the Laplace-Fourier coefficients does not hold for general $p$.

\subsection{The space $\sob{p}{s}{d}$ as a Bessel potential space}

The \emph{Bessel operator of order $s$},
\begin{equation} \label{eq:Bessel operator}
\BeOp{-s} \DEF \left( 1 - \LBo{d} \right)^{s/2}, \qquad s \in \R,
\end{equation}
is a pseudodifferential operator of order $s$ with symbol $(b_\ell^{(s)})_{\ell\geq0}$ given by
\begin{equation} \label{eq:coeff.Bessel}
b_{\ell}^{(s)} \DEF \left( 1 + \lambda_\ell \right)^{s/2} \asymp \left( 1 + \ell \right)^{s}, \qquad \ell = 0, 1, 2, \ldots.
\end{equation}
For $s \geq 0$ it is an operator from $\sob{p}{s}{d}$ to $\Lp{p}{d}$ defined by
\begin{equation} \label{eq:Bessel.Operator}
\BeOp{-s}[f] \DEF \sum_{\ell = 0}^\infty b_\ell^{(s)} \, Y_\ell[f].
\end{equation}
We shall also need the inverse operator $\BeOp{s}: \Lp{p}{d} \to \sob{p}{s}{d}$, which, in contrast to $\BeOp{-s}$ for $s \geq 0$, is a smoothing operator.
The Bessel operator satisfies the following identities:
\begin{equation} \label{eq:Bessel.op.properties}
\BeOp{-\alpha} \BeOp{-\beta} = \BeOp{-(\alpha+\beta)}, \qquad \left( \BeOp{-\alpha} \right)^{-1} = \BeOp{\alpha}, \qquad \BeOp{0} = \mathrm{Id}, \qquad \alpha, \beta \in \R.
\end{equation}
The generalized Sobolev space $\sob{p}{s}{d}$ of Definition~\ref{def:generalized.Sobolev.space} can be interpreted as a \emph{Bessel potential space} and we can use the following equivalent characterization.
\begin{proposition} \label{prop:W.p.s.characterization}
Let $s \geq 0$ and $1 \leq p \leq \infty$. Then $\sob{p}{s}{d}$ is the set of all functions $f \in \Lp{p}{d}$ for which $\BeOp{-s}[f] \in \Lp{p}{d}$, and  $\| f \|_{\sob{p}{s}{d}} = \| \BeOp{-s}[f] \|_{p}$.
\end{proposition}

For $s \geq 0$ we define the zonal \emph{Bessel kernel}
\begin{equation} \label{eq:Bessel.kernel}
\BeKe{s}( \PT{x} \cdot \PT{y} ) \DEF \sum_{\ell = 0}^\infty b_\ell^{(-s)} \, Z(d, \ell) \, P_\ell^{(d)}( \PT{x} \cdot \PT{y} ), \qquad \PT{x}, \PT{y} \in \sph{d}.
\end{equation}
Then we can use the following characterization of $\sob{p}{s}{d}$.

\begin{proposition} \label{prop:Bessel.potential}
Let $s \geq 0$ and $1 \leq p \leq \infty$. Then $f \in \sob{p}{s}{d}$ if and only if $f$ is a \emph{Bessel potential} of a function $g \in \Lp{p}{d}$; that is,
\begin{equation} \label{eq:Bessel.potential}
f( \PT{x} ) = \int_{\sph{d}} \BeKe{s}( \PT{y} \cdot \PT{x} ) \, g( \PT{y} ) \, \dd \sigma_d( \PT{y} ) = \big( \BeKe{s} \ast g \big)( \PT{x} ), \qquad \PT{x} \in \sph{d}.
\end{equation}
Moreover, we have $\BeOp{-s}[f] = g$ and $\| f \|_{\sob{p}{s}{d}} = \| g \|_{p}$.
\end{proposition}

Indeed, any convolution \eqref{eq:Bessel.potential} is in $\sob{p}{s}{d}$ by Young's inequality together with the following boundedness result. (The proof will be postponed until the end of Section~\ref{sec:auxiliary.results}.)

\begin{lemma}[Boundedness of the $\Lp{q}{d}$-norm of the Bessel kernel] \label{lem:q.norm.boundedness.Bessel.kernel}
Let $d \geq 1$, ${1 \leq p, q \leq \infty}$ such that $1/p + 1/q = 1$ and $s > d/p$. Then there exists a constant $c > 0$ such that
\begin{equation}
\big\| \BeKe{s} \big\|_{q,d} \leq 1 + \frac{c}{1 - 2^{d/p-s}}.
\end{equation}
\end{lemma}

We remark that for $p = 2$, the generalized Sobolev space $\sob{p}{s}{d}$ is a reproducing kernel Hilbert space with reproducing kernel $\BeKe{2s}$ (cf. \cite[Sec.~2.4]{BrSaSlWo2014}).
For further reading on Bessel potential spaces, we refer to the classical paper \cite{Str1983} and the more recent paper \cite{HuZae2009}. For the spherical case, we rely on \cite{BCCGST2013}.

\subsection{Embedding results}
\label{subsec:embedding}

%
For the readers convenience, we briefly summarize some relevant embedding results (see, e.g.,  Aubin~\cite{Au:1998}).

\begin{proposition}[Continuous embedding into $C( \sph{d} )$]
\label{prop:embed.Wp}
Let $d \geq 1$. The Sobolev space $\sob{p}{s}{d}$ is continuously embedded into $C( \sph{d} )$ if $s > d/p$.
\end{proposition}


For fixed $p$, smoother Sobolev spaces are included in coarser ones:

\begin{proposition}[Continuous embedding, $p$ fixed]
Let $d \geq 1$. For fixed $p$ with $1 \leq p \leq \infty$, $\sob{p}{s^\prime}{d}$ is continuously embedded into $\sob{p}{s}{d}$ if $0 \leq s < s^\prime < \infty$.
\end{proposition}

%

The standard embedding results for $\mathbb{L}_p$-spaces immediately yield the following embedding of $\sob{p^\prime}{s}{d}$ into $\sob{p}{s}{d}$, $ p < p^\prime$:

\begin{proposition}[Continuous embedding, $s$ fixed]
Let $d \geq 1$. For fixed $s$ with $0 \leq s < \infty$, $\sob{p^\prime}{s}{d}$ is continuously embedded into $\sob{p}{s}{d}$ if $1 \leq p < p^\prime \leq \infty$.
\end{proposition}

\section{Worst-case error and QMC-design sequences for $\sob{p}{s}{d}$}
\label{sec:wce}

\subsection{Worst-case error}

We recall that the definition of worst-case error is given in \eqref{eq:wceB}. Let $\nu_N \DEF \nu[X_N]$ be the atomic measure associated with $X_N = \{ \PT{x}_1,\ldots,\PT{x}_N \}$ that places the point mass $1/N$ at each point in $X_N$; i.e.,
\begin{equation*}
\nu_N = \nu[X_N] = \frac{1}{N} \sum_{j=1}^N \delta_{\PT{x}_{j}}.
\end{equation*}
Then the error of integration of a continuous function $f$ on $\sph{d}$ can be written as
\begin{equation*}
\numint[X_N](f) - \xctint(f) = \int_{\sph{d}} f( \PT{x} ) \, \dd \mu_N( \PT{x} ),
\end{equation*}
with the signed measure $\mu_N$ defined by $\mu_N = \nu_N - \sigma_d$.
For the Sobolev space $\sob{p}{s}{d}$ with $s > d/p$, the worst-case error has the following form in terms of the Bessel kernel: let
\begin{equation} \label{eq:SRB-Bessel}
\BeKeMod{s}( t ) \DEF \BeKe{s}( t ) - 1 = \sum_{\ell=1}^{\infty} b_{\ell}^{(-s)} Z(d,\ell) P_{\ell}^{(d)}( t ), \qquad -1 \leq t \leq 1,
\end{equation}
then the worst-case error is equal to the $\Lp{q}{d}$-norm of the following function,
\begin{equation}\label{eq:SBF-Bessel}
\BeKeNumIntErr{N}{s}( \PT{y} ) \DEF \BeKeNumIntErr{}{s}[X_N]( \PT{y} ) \DEF \frac{1}{N} \sum_{j=1}^{N} \BeKe{s}( \PT{x}_{j} \cdot \PT{y} ) - 1 = \frac{1}{N} \sum_{j=1}^{N} \BeKeMod{s}( \PT{x}_{j} \cdot \PT{y} ), \qquad \PT{y} \in \sph{d}.
\end{equation}
For each fixed $\PT{y} \in \sph{d}$, this function represents the error of numerical integration of the zonal function $\PT{x} \mapsto \BeKe{s}( \PT{x} \cdot \PT{y} )$, $\PT{x} \in \sph{d}$, of the QMC method based on the node set $X_N = \{ \PT{x}_1, \ldots, \PT{x}_N \} \subset \sph{d}$.

\begin{theorem} \label{thm:wce.for.Wp}
Let $d\geq1$, $1 \leq p, q \leq \infty$ with $1/p + 1/q = 1$ and $s > d / p$. Then, for a QMC method $\numint[X_N]$ with node set $X_N = \{ \PT{x}_1,\ldots,\PT{x}_N \} \subset \sph{d}$,
\begin{equation} \label{eq:wce.in.thm}
\WCE(\numint[X_N]; \sob{p}{s}{d} ) = \left\| \int_{\sph{d}} \BeKe{s}( \PT{x} \cdot \PT{\cdot} ) \dd \mu_N( \PT{x} ) \right\|_q = \left\| \BeKeNumIntErr{N}{s} \right\|_q.
\end{equation}
\end{theorem}

\begin{remark*}
In the Hilbert space setting $p = q = 2$, one has the closed form representation
\begin{equation} \label{eq:Hilbert.space.setting.closed.form}
\left\| \BeKeNumIntErr{N}{s} \right\|_2 = \left( \frac{1}{N^2} \sum_{j=1}^N \sum_{k=1}^N \BeKeMod{2s}( \PT{x}_j \cdot \PT{x}_k ) \right)^{1/2},
\end{equation}
which follows from the relation
\begin{equation} \label{eq:BeKe.identity}
\int_{\sph{d}} \BeKe{\alpha}( \PT{x} \cdot \PT{z} ) \BeKe{\beta}( \PT{y} \cdot \PT{z} ) \, \dd \sigma_d( \PT{z} ) = \BeKe{\alpha + \beta}( \PT{x} \cdot \PT{y} ), \qquad \text{$\PT{x}, \PT{y} \in \sph{d}$, $\alpha, \beta > 0$.}
\end{equation}
\end{remark*}

\begin{proof}[Proof of Theorem~\ref{thm:wce.for.Wp}]
First, note that the last expression in \eqref{eq:wce.in.thm} follows from substituting ${\mu_N = \nu_N - \sigma_d}$ into the middle expression in \eqref{eq:wce.in.thm}. Since $s > d / p$, the Sobolev space $\sob{p}{s}{d}$ is continuously embedded into $C( \sph{d} )$ by Proposition~\ref{prop:embed.Wp}, and every element in $\sob{p}{s}{d}$ has a continuous representative. For $f \in \sob{p}{s}{d}$ the following inequality, due to \cite[Corollary~2.4]{BCCGST2013}, can be derived from \eqref{eq:Bessel.potential} together with Fubini's theorem and H{\"o}lder's inequality,
\begin{equation} \label{eq:WCE.upper.bound}
\left| \int_{\sph{d}} f( \PT{x} ) \dd \mu_N( \PT{x} ) \right| \leq \left\{ \int_{\sph{d}} \left| \int_{\sph{d}} \BeKe{s}( \PT{x} \cdot \PT{y} ) \dd \mu_N( \PT{x} ) \right|^q \dd \sigma_d( \PT{y} ) \right\}^{1/q} \left\| f \right\|_{\sob{p}{s}{d}}.
\end{equation}
These integrals are well defined and finite. Therefore,
\begin{equation} \label{eq:WCE.upper.bound.B}
\WCE(\numint[X_N]; \sob{p}{s}{d} ) \leq \left\| \int_{\sph{d}} \BeKe{s}( \PT{x} \cdot \PT{\cdot} ) \dd \mu_N( \PT{x} ) \right\|_q = \left\| \BeKeNumIntErr{N}{s} \right\|_q.
\end{equation}

We complete the proof by constructing a bad function $f_{\mathrm{\mathbf{bad}}}$ with $\| f_{\mathrm{\mathbf{bad}}} \|_{\sob{p}{s}{d}} = 1$ whose absolute integration error is equal to the right-hand side above when $1 \leq q < \infty$, and giving a lower estimate argument for $\WCE(\numint[X_N]; \sob{p}{s}{d} )$ in the case $q = \infty$.
Let $1 \leq q < \infty$ and $1/p + 1/q = 1$. Consider the function $\BeKeNumIntErr{N}{s}$ from \eqref{eq:SBF-Bessel}. As $\BeKeNumIntErr{N}{s} \in \Lp{q}{d}$, there exists a function $u \in \Lp{p}{d}$ such that
\begin{equation*} 
\left\| u \right\|_p = 1 \qquad \text{and} \qquad \left| \int_{\sph{d}} \BeKeNumIntErr{N}{s}( \PT{y} ) \, u( \PT{y} ) \dd \sigma_d( \PT{y} ) \right| = \left\| \BeKeNumIntErr{N}{s} \right\|_q:
\end{equation*}
one can choose
\begin{equation*}
u( \PT{y} ) =
\begin{cases}
\dfrac{\left| \BeKeNumIntErr{N}{s}( \PT{y} ) \right|^{q-1}}{\left\| \BeKeNumIntErr{N}{s} \right\|_q^{q-1}} \, \dfrac{\left| \BeKeNumIntErr{N}{s}( \PT{y} ) \right|}{\BeKeNumIntErr{N}{s}( \PT{y} )} & \text{if $\BeKeNumIntErr{N}{s}( \PT{y} ) \neq 0$,} \\[.5em]
0 & \text{if $\BeKeNumIntErr{N}{s}( \PT{y} ) = 0$,}
\end{cases}
\qquad \PT{y} \in \sph{d}.
\end{equation*}
Now, set $v = \BeOp{s}[ u ]$. Then $v \in \sob{p}{s}{d}$. In fact, by definition of $u$,
\begin{equation*}
\left\| v \right\|_{\sob{p}{s}{d}} = \left\| \BeOp{-s}[ v ] \right\|_p = \left\| \BeOp{-s}[ \BeOp{s}[ u ] ] \right\|_p = \left\| u \right\|_p = 1.
\end{equation*}
The bad function $f_{\mathrm{\mathbf{bad}}}$ with $\| f_{\mathrm{\mathbf{bad}}} \|_{\sob{p}{s}{d}} = 1$ is now the continuous representative of $v$ in $\sob{p}{s}{d}$. Because of the convolution formula $\BeOp{s}[u] = \BeKe{s} \ast u$, we obtain for the absolute error of numerical integration by the QMC method $\numint[X_N]$,
\begin{equation*}
\begin{split}
\left| \int_{\sph{d}} f_{\mathrm{\mathbf{bad}}}( \PT{x} ) \, \dd \mu_N( \PT{x} ) \right|
&= \left| \int_{\sph{d}} \int_{\sph{d}} \BeKe{s}( \PT{x} \cdot \PT{y} ) u( \PT{y} ) \, \dd \sigma_d( \PT{y} ) \dd \mu_N( \PT{x} ) \right| \\
&= \left| \int_{\sph{d}} u( \PT{y} ) \, \BeKeNumIntErr{N}{s}( \PT{y} ) \, \dd \sigma_d( \PT{y} ) \right| = \left\| \BeKeNumIntErr{N}{s} \right\|_q.
\end{split}
\end{equation*}
This lower bound of $\WCE(\numint[X_N]; \sob{p}{s}{d} )$ matches the upper bound in \eqref{eq:WCE.upper.bound.B}.

Let $q = \infty$ (i.e., $p = 1$). By the definition of the $\Lp{\infty}{d}$-norm, to every $\varepsilon > 0$ there exists a subset $E_\varepsilon \subset \sph{d}$ of positive $\sigma_d$-measure such that $| \BeKeNumIntErr{N}{s}( \PT{y} ) | \geq \| \BeKeNumIntErr{N}{s} \|_\infty - \varepsilon$ on $E_\varepsilon$ and a function $u_\varepsilon \in \Lp{1}{d}$ satisfying
\begin{equation*}
\left\| u_\varepsilon \right\|_1 = 1 \qquad \text{and} \qquad \left| \int_{\sph{d}} \BeKeNumIntErr{N}{s}( \PT{y} ) \, u_\varepsilon( \PT{y} ) \dd \sigma_d( \PT{y} ) \right| \geq \left\| \BeKeNumIntErr{N}{s} \right\|_\infty - \varepsilon.
\end{equation*}
One can choose
\begin{equation*}
u_\varepsilon( \PT{y} ) =
\begin{cases}
\dfrac{\chi_\varepsilon( \PT{y} )}{\sigma_d( E_\varepsilon )} \, \dfrac{\left| \BeKeNumIntErr{N}{s}( \PT{y} ) \right|}{\BeKeNumIntErr{N}{s}( \PT{y} )} & \text{if $\BeKeNumIntErr{N}{s}( \PT{y} ) \neq 0$,} \\[.5em]
0 & \text{if $\BeKeNumIntErr{N}{s}( \PT{y} ) = 0$,}
\end{cases}
\qquad \PT{y} \in \sph{d},
\end{equation*}
where $\chi_\varepsilon \DEF \chi_{E_\varepsilon}$ is the characteristic function of the set $E_\varepsilon$. Similarly as before, one shows that $v_\varepsilon = \BeOp{s}[u_\varepsilon]$ is in $\sob{1}{s}{d}$ and $\| v_\varepsilon \|_{\sob{1}{s}{d}} = 1$. Taking $f_{\mathrm{\mathbf{bad}},\varepsilon}$ to be the continuous representative of $v_\varepsilon$ in $\sob{1}{s}{d}$, we arrive at
\begin{equation*}
\WCE(\numint[X_N]; \sob{1}{s}{d} ) \geq \left| \int_{\sph{d}} f_{\mathrm{\mathbf{bad}},\varepsilon}( \PT{x} ) \, \dd \mu_N( \PT{x} ) \right| = \left| \int_{\sph{d}} \BeKeNumIntErr{N}{s}( \PT{y} ) \, u_\varepsilon( \PT{y} ) \dd \sigma_d( \PT{y} ) \right| \geq \left\| \BeKeNumIntErr{N}{s} \right\|_\infty - \varepsilon.
\end{equation*}
Since $\varepsilon > 0$ is arbitrary, we have
\begin{equation*}
\WCE(\numint[X_N]; \sob{1}{s}{d} ) \geq \left\| \BeKeNumIntErr{N}{s} \right\|_\infty.
\end{equation*}
The result follows.
\end{proof}

\subsection{WCE Inequalities}\label{subsec:WCE.inequality}

The following property holds for the worst-case error of a QMC method for generalized Sobolev spaces with the same $s$ but different $p$.

\begin{theorem} \label{thm:embed.wce}
Let $d\geq1$, $1 \leq p < p^\prime \leq \infty$ and $s > d / p$. For every $N$-point set $X_N \subset \sph{d}$,
\begin{equation} \label{eq:wce Wp-1}
\WCE(\numint[X_N]; \sob{p^\prime}{s}{d} ) \leq \WCE(\numint[X_N]; \sob{p}{s}{d} ).
\end{equation}
\end{theorem}

\begin{proof}
Let $1 \leq p < p^\prime \leq \infty$ and $s > d / p$. Then one has the continuous embedding inclusions $\sob{p^\prime}{s}{d} \subset \sob{p}{s}{d} \subset C( \sph{d} )$ and, in particular, $\| f \|_{\sob{p}{s}{d}} \leq c \, \| f \|_{\sob{p^\prime}{s}{d}}$ with $c = 1$ because of $\int_{\sph{d}} \dd \sigma_d = 1$ (Proposition~\ref{prop:W.p.s.characterization} and Jensen's inequality). Thus, the unit ball in $\sob{p}{s}{d}$ is larger than the one in $\sob{p^\prime}{s}{d}$ and the result follows from \eqref{eq:wceB}.
\end{proof}

As a consequence of Theorem~\ref{thm:embed.wce} we provide the following proof.

\begin{proof}[Proof of Theorem~\ref{thm:QMC.desing.property.A}]
Let $( X_N )$ be a QMC-design sequence for $\sob{p}{s}{d}$, where $1 \leq p < \infty$ and $s > d/p$. Then, there exists a constant $c > 0$ such that $\WCE(\numint[X_N]; \sob{p}{s}{d} ) \leq c \, N^{-s/d}$ for all $X_N$. Suppose $p < p^\prime \leq \infty$. Then by Theorem~\ref{thm:embed.wce},
\begin{equation*}
\WCE(\numint[X_N]; \sob{p^\prime}{s}{d} ) \leq \WCE(\numint[X_N]; \sob{p}{s}{d} ) \leq c \, N^{-s/d} \qquad \text{for all $X_N$}.
\end{equation*}
Hence by Definition~\ref{def:specific_approx.sph.design}, $( X_N )$ is a QMC-design sequence for $\sob{p^\prime}{s}{d}$.
\end{proof}

Next, we consider worst-case error interrelations for generalized Sobolev spaces with the same $p$ but different $s$. In the Hilbert space setting $p = 2$ the reproducing kernel Hilbert space method gives a heat kernel representation of the worst-case error which leads to the following result.

\begin{proposition}[{\cite[Lemma~26]{BrSaSlWo2014}}] \label{prop:WCE.inequality.p.EQ.2}
Let $d\geq1$ and $s > d/2$. If $\WCE(\numint[X_N]; \sob{2}{s}{d} ) < 1$, then
\begin{equation} \label{eq:wce for W2^s}
\WCE(\numint[X_N]; \sob{2}{s^\prime}{d} ) < c_{d,s,s^\prime} \left[ \WCE(\numint[X_N]; \sob{2}{s}{d} ) \right]^{s^\prime / s}, \qquad d/2 < s^\prime < s,
\end{equation}
where $c_{d,s,s^\prime} > 0$ depends on the norms for $\sob{2}{s}{d}$ and
$\sob{2}{s^\prime}{d}$, but is independent of~$N$.
\end{proposition}

The proof of Theorem~\ref{thm:QMC.desing.property.B} is based on the following $\Lp{q}{d}$-Bernstein type inequality.

\begin{lemma} \label{lem:Bernstein.inequality} 
Let $d \geq 1$, $1 \leq p, q \leq \infty$ with $1/p + 1/q = 1$ and $s - d/p > \tau > 0$. Then the function $\BeKeNumIntErr{N}{s}$ for an $N$-point set $X_N = \{ \PT{x}_1,\ldots,\PT{x}_N \} \subset \sph{d}$ with mesh ratio $\MESHRATIO( X_N )$ satisfies
\begin{equation} \label{eq:Bernstein.inequality}
\left\| \BeKeNumIntErr{N}{s} \right\|_{\sob{q}{\tau}{d}} \leq c \left[ \MESHRATIO( X_N ) \right]^{d/p} N^{\tau/d} \left\| \BeKeNumIntErr{N}{s} \right\|_{q},
\end{equation}
where $c\ge 1$ depends only on $d$, $p$ and $q$, $s$ and $\tau$.
\end{lemma}

We will provide a proof of \eqref{eq:Bernstein.inequality} in Section~\ref{sec:Bernstein.inequality}.

\begin{remark*}
Mhaskar et al.~\cite[Theorem~6.1, p.~1669]{MhNaPrWa2010} prove \eqref{eq:Bernstein.inequality} for quasi-uniform $X_{N}$. Our estimate holds for general sequences $( X_N )$ but is specific to the kernel $\BeKeNumIntErr{N}{s}$.
An essential feature of \eqref{eq:Bernstein.inequality} is the explicit dependence on the mesh ratio of the point set. This is of importance for determining the stability and error estimates, and thus is of independent interest.
\end{remark*}

\begin{proof}[Proof of Theorem~\ref{thm:QMC.desing.property.B}]
First, we note that for fixed $\PT{x} \in \sph{d}$ the function ${\phi^{(s^\prime)}( \PT{y} ) \DEF \BeKe{s^\prime}(\PT{x}\cdot \PT{y})}$, $\PT{y} \in \sph{d}$, is in $\Lp{q}{d}$ for $s^\prime > d/p$ by Lemma~\ref{lem:q.norm.boundedness.Bessel.kernel}. Then the identity~\eqref{eq:BeKe.identity} (with $\alpha = s^\prime$ and $\beta = s - s^\prime$) gives
\begin{equation*}
\phi^{(s)}( \PT{y} ) = \int_{\sph{d}} \BeKe{s-s^\prime}( \PT{z} \cdot \PT{y} ) \, \phi^{(s^\prime)}( \PT{z} ) \, \dd \sigma_d( \PT{z} ), \qquad \PT{y} \in \sph{d}.
\end{equation*}
Consequently, $\BeKeNumIntErr{N}{s}$ given in \eqref{eq:SBF-Bessel} is the Bessel potential of $\BeKeNumIntErr{N}{s^\prime} \in \Lp{q}{d}$ in the sense of Proposition~\ref{prop:Bessel.potential}. Hence, by Theorem~\ref{thm:wce.for.Wp}, Proposition~\ref{prop:Bessel.potential} and Lemma~\ref{lem:Bernstein.inequality} (with $\tau = s - s^\prime$),
\begin{align}
\WCE(\numint[X_N]; \sob{p}{s^\prime}{d} ) = \left\| \BeKeNumIntErr{N}{s^\prime} \right\|_q
&= \left\| \BeKeNumIntErr{N}{s} \right\|_{\sob{q}{s-s^\prime}{d}} \\
&\leq c \left[ \MESHRATIO( X_N ) \right]^{d/p} N^{( s - s^\prime )/d} \left\| \BeKeNumIntErr{N}{s} \right\|_q \notag \\
&= c \left[ \MESHRATIO( X_N ) \right]^{d/p} N^{( s - s^\prime )/d} \, \WCE(\numint[X_N]; \sob{p}{s}{d} ), \notag
\end{align}
where the constant $c$ depends on $d$, $s$, $s^\prime$, $p$. This completes the proof.
\end{proof}

\section{Filtered Bessel kernel and proof of Lemma~\ref{lem:q.norm.boundedness.Bessel.kernel}}
\label{sec:auxiliary.results}

In this section we use a filtered Bessel kernel to prove Lemma~\ref{lem:q.norm.boundedness.Bessel.kernel}.
Let $d\geq1$, $s \in \R$. Given a filter $h$ (i.e., a smooth function on $\R_+$ with compact support), we define the \emph{filtered Bessel kernel}
\begin{equation} \label{eq:filtered.Bessel.kernel}
\filBeKe{h}{s}( T; \PT{x} \cdot \PT{y} ) \DEF \sum_{\ell = 0}^\infty h\big( \frac{\ell}{T} \big) \, b_{\ell}^{(-s)} \, Z(d,\ell) P^{(d)}_{\ell}(\PT{x} \cdot \PT{y}), \qquad T \geq1, \, \PT{x}, \PT{y} \in \sph{d}.
\end{equation}

In the special case $s = 0$, so $b_{\ell}^{(-s)} = 1$, the following results are known from \cite{NaPeWa2006-2}. 
The more general filtered Bessel kernel in \eqref{eq:filtered.Bessel.kernel} satisfies the following localization estimate. 


\begin{proposition}[{Localized upper bound; cf.~\cite[Lemma~2.8]{BCCGST2013}}]\label{lm:localised.UB}
Let $h$ be a filter with support~$[1/2,2]$. For every positive integer~$n$, there exists a constant $c_n > 0$ such that for every $T > 1$ and $s \geq 0$,
\begin{equation} \label{eq:filtered.Bessel.kernel.loc.estimate}
\left| \filBeKe{h}{s}( T; \PT{x} \cdot \PT{y} ) \right| \leq c_n \, \frac{T^{d-s}}{\left( 1 + T^2 \left| \PT{x} - \PT{y} \right|^2 \right)^{n/2}}, \qquad \PT{x}, \PT{y} \in \sph{d}.
\end{equation}
\end{proposition}

Note that the upper bound is a zonal function, since $| \PT{x} - \PT{y} |^2 = 2 - 2 \PT{x} \cdot \PT{y}$ for $\PT{x}, \PT{y} \in \sph{d}$. The localized upper bound gives the following estimate in which $s > d/p$ to ensure that we are dealing with continuous functions.

\begin{lemma}[$\Lp{q}{d}$-norm of filtered Bessel kernel] \label{lem:q.norm.filtered.Bessel.kernel}
Let $d \geq 1$, $1 \leq p, q \leq \infty$ with $1/p + 1/q = 1$ and $s > d/p$. Suppose $h$ is a filter with support $[1/2,2]$. Then there exists a constant $c > 0$ such that
\begin{equation} \label{eq:filtered.Bessel.kernel.q.norm.estimate}
\big\| \filBeKe{h}{s}( T; \cdot ) \big\|_{q,d} \leq c \, T^{d/p - s}, \qquad T \geq 1.
\end{equation}
The constant $c$ depends only on $h$, $d$, $s$ and $q$.
\end{lemma}
\begin{proof}
First, let $1 \leq q < \infty$. Then
\begin{equation*}
\big\| \filBeKe{h}{s}( T; \cdot ) \big\|_{q,d}^q = \frac{\omega_{d-1}}{\omega_d} \int_{-1}^1 \big| \filBeKe{h}{s}( T; t ) \big|^q \left( 1 - t^2 \right)^{d/2-1} \dd t.
\end{equation*}
The change of variable $2u = 1 + t$ and the localized estimate~\eqref{eq:filtered.Bessel.kernel.loc.estimate} give
\begin{equation*}
\big\| \filBeKe{h}{s}( T; \cdot ) \big\|_{q,d}^q \leq G( T ) \DEF 2^{d-1} \frac{\omega_{d-1}}{\omega_d} \int_0^1 \frac{c_n^q \, T^{q ( d - s )}}{\left( 1 + 4 T^2 - 4 T^2 \, u \right)^{q n / 2}} \, u^{d/2-1} \left( 1 - u \right)^{d/2-1} \dd u
\end{equation*}
for $T > 1$ and positive integers $n$. Rewriting the integral as
\begin{equation*}
G( T ) = \left[ \frac{c_n \, T^{d-s}}{\left( 1 + 4 T^2 \right)^{n/2}} \right]^q 2^{d-1} \frac{\omega_{d-1}}{\omega_d} \int_0^1 \frac{u^{d/2-1} \left( 1 - u \right)^{d/2-1}}{\left( 1 - \frac{4T^2}{1 + 4 T^2} \, u \right)^{q n/2}} \, \dd u,
\end{equation*}
we express $G(T)$ in terms of a Gauss hypergeometric function (cf.~\cite[Eq.~15.6.1]{NIST:DLMF})
\begin{equation*}
G( T ) = \left[ \frac{c_n \, T^{d-s}}{\left( 1 + 4 T^2 \right)^{n/2}} \right]^q \Hypergeom{2}{1}{q n/2, d/2}{d}{\frac{4T^2}{1 + 4 T^2}}.
\end{equation*}
A linear transformation of hypergeometric functions~\cite[last of Eq.~15.8.1]{NIST:DLMF} yields
\begin{equation*}
G( T ) = \left[ \frac{c_n \, T^{d-s}}{\left( 1 + 4 T^2 \right)^{n/2}} \right]^q \left( \frac{1}{1 + 4 T^2} \right)^{d/2 - q n/2} \Hypergeom{2}{1}{d - q n/2, d/2}{d}{\frac{4T^2}{1 + 4 T^2}}.
\end{equation*}
Now choose $n$ to be a fixed integer satisfying $n > 2 d / q$. Because $d - q n / 2 < 0$, the hypergeometric function part is strictly monotonically decreasing on $[0,\infty)$ as a function of $T$. This can be seen from the integral representation (cf. \cite[Eq.~15.6.1]{NIST:DLMF}) of the hypergeometric function and the fact that $4 T^2 / ( 1 + 4 T^2 )$ is strictly increasing on $[0,\infty)$.
Then
\begin{equation*}
\Hypergeom{2}{1}{d - q n/2, d/2}{d}{\frac{4T^2}{1 + 4 T^2}}
\leq \Hypergeom{2}{1}{d - q n/2, d/2}{d}{0} = 1.
\end{equation*}
We arrive at
\begin{equation*} 
\big\| \filBeKe{h}{s}( T; \cdot ) \big\|_{q,d} \leq c_n \frac{T^{d-s}}{\left( 1 + 4 T^2 \right)^{d / ( 2 q )}} \leq \frac{c_n}{2^{d/q}} T^{d ( 1 - 1/q ) - s} \qquad \text{for $T \geq 1$.}
\end{equation*}
The result follows for $1 \leq q < \infty$.

Let $q = \infty$. Using the localized estimate~\eqref{eq:filtered.Bessel.kernel.loc.estimate}, we get for any fixed positive integer $n$,
\begin{equation*} 
\big\| \filBeKe{h}{s}( T; \cdot ) \big\|_{\infty,d} = \max_{-1 \leq t \leq 1} \left| \filBeKe{h}{s}( T; t ) \right| \leq \max_{-1 \leq t \leq 1} \frac{c_n \, T^{d-s}}{\left( 1 + 2 T^2 - 2 T^2 \, t \right)^{n/2}} = c_n \, T^{d-s}
\end{equation*}
for $T \geq 1$. This completes the proof.
\end{proof}

In order to show that the $\Lp{q}{d}$-norm of the zonal Bessel kernel is bounded, we now strengthen the requirement on the filter $h$ with support $[1/2,2]$ occurring in the filtered Bessel kernel \eqref{eq:filtered.Bessel.kernel}, by assuming that
\begin{equation}\label{eq:partition.of.unity.property-1}
h( 2 t ) + h( t ) = 1 \qquad \text{on $[1/2,1]$.}
\end{equation}
This condition is equivalent to saying that $h$ has the \emph{partition of unity property} (see~\cite{NaPeWa2006}), namely
\begin{equation} \label{eq:partition.of.unity.property}
\sum_{m=0}^\infty h\big( \frac{x}{2^m} \big) = 1 \qquad \text{for all $x \geq 1$.}
\end{equation}

\begin{proof}[Proof of Lemma~\ref{lem:q.norm.boundedness.Bessel.kernel}]
Let $h$ be a filter with support $[1/2,2]$ and the partition of unity property.
Then using \eqref{eq:filtered.Bessel.kernel} and \eqref{eq:partition.of.unity.property}, we get

\begin{equation*}
\sum_{m=1}^\infty \filBeKe{h}{s}( 2^{m-1}; t ) = \sum_{\ell = 1}^\infty \left( \sum_{m=1}^\infty h\big( \frac{\ell}{2^{m-1}} \big) \right) \, b_{\ell}^{(-s)} \, Z(d,\ell) P^{(d)}_{\ell}( t ) = \BeKe{s}( t ) - 1, \quad -1 \leq t \leq 1.
\end{equation*}
Then for $s > d/p$ the triangle inequality and the filtered Bessel kernel estimate~\eqref{eq:filtered.Bessel.kernel.q.norm.estimate} yield
\begin{equation*}
\Big\| \BeKe{s} - 1 \Big\|_{q,d} = \Big\| \sum_{m=1}^\infty \filBeKe{h}{s}( 2^{m-1}; \cdot ) \Big\|_{q,d} \leq \sum_{m=0}^\infty \Big\| \filBeKe{h}{s}( 2^{m}; \cdot ) \Big\|_{q,d} \leq c \sum_{m = 0}^\infty \left( 2^m \right)^{d/p-s} =  \frac{c}{1 - 2^{d/p-s}},
\end{equation*}
where $c$ is the constant in Lemma~\ref{lem:q.norm.filtered.Bessel.kernel}. On observing that $\| 1 \|_{q,d} = 1$, we get
\begin{equation*}
\big\| \BeKe{s} \big\|_{q,d} \leq \Big\| \BeKe{s} - 1 \Big\|_{q,d} + \| 1 \|_{q,d} \leq 1 + \frac{c}{1 - 2^{d/p-s}}.
\end{equation*}
This completes the proof.
\end{proof}

\section{Proof of Lemma~\ref{lem:Bernstein.inequality}}
\label{sec:Bernstein.inequality}

In this section, we prove the Bernstein type inequality \eqref{eq:Bernstein.inequality} for the function
\begin{equation*}
\BeKeNumIntErr{N}{s}( \PT{y} ) = \frac{1}{N} \sum_{j=1}^{N} \BeKeMod{s}( \PT{x}_{j} \cdot \PT{y} ), \qquad \PT{y} \in \sph{d},
\end{equation*}
which is central to the computation of the worst-case error (see Theorem~\ref{thm:wce.for.Wp}).

In order to establish this result, we make use of the well-known Bernstein inequality for spherical polynomials (see \cite[Proposition~4.3]{LeMh2006} and \cite[Theorem~2]{Ru1992}).

\begin{proposition} \label{prop:Bernstein.inequality.for.polynomials}
For $d \geq 1$, $1 \leq q \leq \infty$ and $\tau \geq 0$, there holds 
\begin{equation} \label{eq:Bernstein.inequality.polynomial}
\left\| P \right\|_{\sob{q}{\tau}{d}} \leq c_{q,\tau,d} \, n^\tau \left\| P \right\|_{q}, \qquad P \in \Pi_n^d,
\end{equation}
where $\Pi_n^d$ denotes the family of spherical polynomials on $\sph{d}$ of degree at most $n$.
\end{proposition}

We follow the general approach of \cite{MhNaPrWa2010}, but with the crucial difference that we are able to replace a positive-definite assumption in  \cite{MhNaPrWa2010} by the precise lower bound
\begin{equation} \label{eq:auxiliary.estimates}
\left\| \BeKeNumIntErr{N}{s} \right\|_q = \WCE(\numint[X_N]; \sob{p}{s}{d} ) \geq \frac{c_{p,s,d}^\prime}{N^{s/d}} > 0, \qquad s > \frac{d}{p}, \quad \frac{1}{p} + \frac{1}{q} = 1,
\end{equation}
which follows from Theorem~\ref{thm:wce.for.Wp} and the lower bound of Proposition~\ref{prop:generic.lower.bound}.
Our strategy is to approximate $\BeKeNumIntErr{N}{s}$ by spherical polynomials on $\sph{d}$ of degree~$2^m \asymp N^{1/d}$ that are convolution approximations of~$\BeKeNumIntErr{N}{s}$ with filtered Bessel kernels. For a smooth filter $\widetilde{h}$ with support $[0, 2]$, to be specified below, we define (see \eqref{eq:filtered.Bessel.kernel} and \eqref{eq:coeff.Bessel})
\begin{equation*}
\eta_0 \equiv 1, \qquad \eta_m \DEF \eta_{m,\widetilde{h}} \DEF \filBeKe{\widetilde{h}}{0}( 2^{m-1}; \cdot ), \quad m \geq 1.
\end{equation*}
Then it can be readily seen that $\eta_m \ast \BeKeNumIntErr{N}{s}$ is a spherical polynomial of degree $2^m - 1$.
By \eqref{eq:BeKe.identity}, $\BeKeNumIntErr{N}{s}$ is the Bessel potential of $\BeKeNumIntErr{N}{s-\tau} \in \Lp{q}{d}$ for $s - d/p > \tau > 0$, so $\BeKeNumIntErr{N}{s} \in \sob{q}{\tau}{d}$. The triangle inequality then gives
\begin{equation} \label{eq:strategy.inequality}
\left\| \BeKeNumIntErr{N}{s} \right\|_{\sob{q}{\tau}{d}} \leq \left\| \eta_m \ast \BeKeNumIntErr{N}{s} \right\|_{\sob{q}{\tau}{d}} + \left\| \BeKeNumIntErr{N}{s} - \eta_m \ast \BeKeNumIntErr{N}{s} \right\|_{\sob{q}{\tau}{d}}.
\end{equation}
From \eqref{eq:Bernstein.inequality.polynomial} and \eqref{eq:Young.s.ineq.A} we deduce the following bound for the polynomial part:
\begin{equation} \label{eq:strategy.inequality.polynomial.part}
\left\| \eta_m \ast \BeKeNumIntErr{N}{s} \right\|_{\sob{q}{\tau}{d}} \leq  c_{q,\tau,d} \, 2^{m \tau} \left\| \eta_m \ast \BeKeNumIntErr{N}{s} \right\|_q \leq  c_{q,\tau,d} \, 2^{m \tau}  \left\| \eta_m \right\|_{1,d} \left\| \BeKeNumIntErr{N}{s} \right\|_q.
\end{equation}

The challenging part is to control the error of approximation $\| \BeKeNumIntErr{N}{s} - \eta_m \ast \BeKeNumIntErr{N}{s} \|_{\sob{q}{\tau}{d}}$.
For this purpose we decompose the convolution of $\eta_m$ against a function $f \in \Lp{q}{d}$,
\begin{equation} \label{eq:master.identity}
\eta_m \ast f = \sum_{k=0}^m \filterpsi{k}{} \ast f, \qquad \filterpsi{0}{} \equiv 1, \quad \filterpsi{k}{} \DEF \filterpsi{k,h}{} \DEF \filBeKe{h}{0}( 2^{k-1}; \cdot ), \quad k \geq 1,
\end{equation}
where $h$ is a filter with support $[1/2,2]$ and range $[0,1]$ that also has the property \eqref{eq:partition.of.unity.property-1}.
We now specify $\widetilde{h}$ in terms of $h$ as follows\footnote{The partition of unity property implies smoothness at the transition point $t = 1$. The requirement that $\widetilde{h}$ is $1$ on $[0,1]$ implies that convolution with $\eta_m$ reproduces a spherical polynomial of degree $\leq 2^{m-1}$.}
\begin{equation*}
\widetilde{h}( t ) \DEF \begin{cases} 1 & \text{if $t \in [0,1]$,} \\ h( t ) & \text{if $t \geq 1$.} \end{cases}
\end{equation*}
Then it can be readily verified that
\begin{equation*}
\widetilde{h}\big( \frac{\ell}{2^{m-1}} \big) = \sum_{k=1}^m h\big( \frac{\ell}{2^{k-1}} \big), \qquad \ell, m \geq 1,
\end{equation*}
which in turn implies \eqref{eq:master.identity}.
Furthermore, we note that by \cite[Lemma~2.11]{NaPeWa2006} 
\begin{equation} \label{eq:decomposition}
\sum_{k=0}^m \filterpsi{k}{} \ast f \to f \qquad \text{as $m \to \infty$ in $\Lp{q}{d}$.}
\end{equation}

Now, let $s - d / p > \tau > 0$ and for $\PT{x} \in \sph{d}$ set $\phi^{(s)}( \PT{y} ) \DEF \BeKeMod{s}( \PT{x} \cdot \PT{y} )$. First, observe that
\begin{equation*}
\BeOp{-\tau}[ \phi^{(s)} ] = \phi^{(s-\tau)}, \qquad \BeOp{-\tau}[ \eta_m \ast \phi^{(s)} ] = \eta_m \ast \BeOp{-\tau}[ \phi^{(s)} ] = \eta_m \ast \phi^{(s-\tau)}.
\end{equation*}
By linearity, these relations also hold for $\BeKeNumIntErr{N}{s}$. Hence, by Proposition~\ref{prop:W.p.s.characterization},
\begin{equation*}
\left\| \BeKeNumIntErr{N}{s} - \eta_m \ast \BeKeNumIntErr{N}{s} \right\|_{\sob{q}{\tau}{d}} = \left\| \BeOp{-\tau}[ \BeKeNumIntErr{N}{s} ] - \BeOp{-\tau}[ \eta_m \ast \BeKeNumIntErr{N}{s} ] \right\|_{q} = \left\| \BeKeNumIntErr{N}{s-\tau} - \eta_m \ast \BeKeNumIntErr{N}{s-\tau} \right\|_{q}.
\end{equation*}
Application of the decomposition relations~\eqref{eq:master.identity} and \eqref{eq:decomposition} and the triangle inequality gives
\begin{equation} \label{eq:Bernstein.estimate.PRE}
\left\| \BeKeNumIntErr{N}{s} - \eta_m \ast \BeKeNumIntErr{N}{s} \right\|_{\sob{q}{\tau}{d}} = \left\| \sum_{k=m+1}^\infty \filterpsi{k}{} \ast \BeKeNumIntErr{N}{s-\tau} \right\|_{q} \leq \sum_{k=m+1}^\infty \left\| \filterpsi{k}{} \ast \BeKeNumIntErr{N}{s-\tau} \right\|_{q}.
\end{equation}
Defining
\begin{equation} \label{eq:filter.phi.s}
\filterpsi{k}{(s)} \DEF \filterpsi{k}{} \ast \phi^{(s)} = \filBeKe{h}{s}( 2^{k-1}; \cdot ), \qquad k \geq 1,
\end{equation}
we deduce
\begin{equation*}
\left\| \filterpsi{k}{} \ast \BeKeNumIntErr{N}{s-\tau} \right\|_{q} = \left\| \frac{1}{N} \sum_{j=1}^N \filterpsi{k}{(s-\tau)}( \PT{x}_j \cdot \cdot ) \right\|_{q} \leq \frac{1}{N} \sum_{j=1}^N \left\| \filterpsi{k}{(s-\tau)}( \PT{x}_j \cdot \cdot ) \right\|_{q} = \left\| \filBeKe{h}{s-\tau}( 2^{k-1}; \cdot ) \right\|_{q,d}.
\end{equation*}
Lemma~\ref{lem:q.norm.filtered.Bessel.kernel} then yields (since $s - \tau > d/p$ and $2^m \asymp N^{1/d}$)
\begin{equation}
\begin{split} \label{eq:Bernstein.estimat.q.EQ.1}
\left\| \BeKeNumIntErr{N}{s} - \eta_m \ast \BeKeNumIntErr{N}{s} \right\|_{\sob{q}{\tau}{d}}
&\leq c^\prime \, \sum_{k=m+1}^\infty ( 2^k )^{d/p-(s - \tau)} \\
&\leq c^{\prime\prime} \, (2^m)^{d/p - (s-\tau)} \leq c^{\prime\prime\prime} \, N^{1/p-(s-\tau) / d}.
\end{split}
\end{equation}
The upper bound in \eqref{eq:Bernstein.estimat.q.EQ.1} is not strong enough to give the result in Lemma~\ref{lem:Bernstein.inequality} except in the case $q = 1$.
The following result will enable us to settle the other extremal case $q = \infty$; however, it requires geometric information about the point set.

\begin{lemma} \label{lem:representer.estimate}
Let $s^{\prime} > d \geq 1$. Then there is a constant $c$ such that for every point set $X_N = \{ \PT{x}_1, \dots, \PT{x}_N \} \subset \sph{d}$,
\begin{equation*}
\left| \frac{1}{N} \sum_{j=1}^N \filterpsi{k}{(s^{\prime})}( \PT{x}_j \cdot \PT{y} ) \right| \leq c \left[ \MESHRATIO( X_N ) \right]^d N^{-1} \, 2^{-k ( s^{\prime} - d )}, \qquad \PT{y} \in \sph{d}; \, k = m, m + 1, m + 2, \dots,
\end{equation*}
where $m \DEF \lfloor \frac{1}{d} \log_2 N \rfloor$ and $\filterpsi{k}{(s^{\prime})}$ is given in \eqref{eq:filter.phi.s}.
\end{lemma}

\begin{proof}
The point set $X_N$ uniquely determines a Voronoi cell decomposition $\{ R_1, \dots, R_N \}$ of $\sph{d}$ with $\PT{x}_j \in R_j$. It has the property that $\min_{1 \leq j \leq N} \sigma_d( R_j ) \geq \beta_d \, [ \SEP( X_N ) ]^d$ for some constant $\beta_d$ depending only on $d$. Utilizing a Marcinkiewicz-Zygmund type inequality from \cite[Corollary~4.6]{MhNaPrWa2010},
\begin{equation*}
\left| \left\| \filterpsi{k}{(s^{\prime})} \right\|_{1,d} - \sum_{j=1}^N \sigma_d( R_j ) \left| \filterpsi{k}{(s^{\prime})}( \PT{x}_j \cdot \PT{y} ) \right| \right| \leq c^\prime \left[ 2^k \COV( X_N ) \right]^d \Eb{2^{k-1}}{\BeKeMod{s^{\prime}}}{}, \qquad \PT{y} \in \sph{d},
\end{equation*}
where $\Eb{n}{f}{} \DEF \inf_{P \in \Pi_n^d} \left\| f - P \right\|_1$ is the \emph{error of best $\Lp{1}{d}$-approximation} by spherical polynomials on $\sph{d}$ of degree at most $n$, we obtain
\begin{align*}
\left| \sum_{j=1}^N \filterpsi{k}{(s^{\prime})}( \PT{x}_j \cdot \PT{y} ) \right|
&\leq \frac{1}{\min_{1 \leq j \leq N} \sigma_d( R_j )} \sum_{j=1}^N \sigma_d( R_j ) \left| \filterpsi{k}{(s^{\prime})}( \PT{x}_j \cdot \PT{y} ) \right| \\
&\leq \frac{1}{\min_{1 \leq j \leq N} \sigma_d( R_j )} \left( \left\| \filterpsi{k}{(s^{\prime})} \right\|_{1,d} + c^\prime \left[ 2^k \COV( X_N ) \right]^d \Eb{2^{k-1}}{\BeKeMod{s^{\prime}}}{} \right).
\end{align*}
Now, by Lemma~\ref{lem:q.norm.filtered.Bessel.kernel},
\begin{equation*}
\left\| \filterpsi{k}{(s^{\prime})} \right\|_{1,d} = \left\| \filBeKe{h}{s^{\prime}}( 2^{k-1}; \cdot ) \right\|_{1,d} \leq c^{\prime\prime} \, 2^{-k s^{\prime}}
\end{equation*}
and proceeding similarly to the derivation of \eqref{eq:Bernstein.estimat.q.EQ.1}, we get
\begin{equation*}
\Eb{2^{k-1}}{\BeKeMod{s^{\prime}}}{} \leq \left\| \BeKeMod{s^{\prime}} - \eta_{k-1} \ast \BeKeMod{s^{\prime}} \right\|_{1,d} \leq c^{\prime\prime\prime} \, 2^{-k s^{\prime}}.
\end{equation*}
Hence
\begin{align*}
\left| \sum_{j=1}^N \filterpsi{k}{(s^{\prime})}( \PT{x}_j \cdot \PT{y} ) \right|
&\leq \frac{c^{\prime\prime} \, 2^{-k s^{\prime}} + c^\prime \left[ 2^k \COV( X_N ) \right]^d c^{\prime\prime\prime} \, 2^{-k s^{\prime}}}{\beta_d \, [ \SEP( X_N ) ]^d} \\
&= \frac{1}{\beta_d} \left[ \frac{\COV( X_N )}{\SEP( X_N )} \right]^d \left(  \frac{c^{\prime\prime}}{\left[ 2^k \COV( X_N ) \right]^d} + c^\prime \, c^{\prime\prime\prime} \right) 2^{-k( s^{\prime} - d )}.
\end{align*}
The parenthetical expression is bounded because (recalling $k \geq m$) $2^{-k} \leq 2^{-m} \asymp N^{-1/d} \leq c^{iv} \COV( X_N )$.
\end{proof}

\begin{lemma} \label{lem:filtered.convolution.approx.error} Let $d\geq1$, $1 \leq p,q \leq \infty$ with $1/p + 1/q = 1$ and $s - d / p > \tau > 0$, and $X_N$ an $N$-point set on~$\sph{d}$. Let $m = \lfloor \frac{1}{d} \log_2 N \rfloor$. Then we have
\begin{equation}\label{eq:fil approx phiN-1}
\left\| \BeKeNumIntErr{N}{s}-\eta_{m}\ast \BeKeNumIntErr{N}{s} \right\|_{\sob{q}{\tau}{d}}\leq c\: \left[ \MESHRATIO( X_N ) \right]^{d/p} N^{- (s-\tau)/d},
\end{equation}
where the constant $c$ depends only on $d$, $p$, $s$ and $\tau$.
\end{lemma}

\begin{proof}
The case $q = 1$ is given by \eqref{eq:Bernstein.estimat.q.EQ.1}.
It suffices to consider the case $q = \infty$, for then the case $1 < q < \infty$ follows from the Riesz-Thorin theorem. By Lemma~\ref{lem:representer.estimate},
\begin{equation*}
\left\| \filterpsi{k}{} \ast \BeKeNumIntErr{N}{s-\tau} \right\|_{\infty} = \sup_{\PT{y} \in \sph{d}} \left| \frac{1}{N} \sum_{j=1}^N \filterpsi{k}{(s-\tau)}( \PT{x}_j \cdot \PT{y} ) \right| \leq c \left[ \MESHRATIO( X_N ) \right]^d N^{-1} \, 2^{-k ( s - \tau - d )}, \qquad k \geq m,
\end{equation*}
and substitution into \eqref{eq:Bernstein.estimate.PRE} gives as before
\begin{align*}
\left\| \BeKeNumIntErr{N}{s} - \eta_m \ast \BeKeNumIntErr{N}{s} \right\|_{\sob{\infty}{\tau}{d}}
&\leq \sum_{k=m+1}^\infty \left\| \filterpsi{k}{} \ast \BeKeNumIntErr{N}{s-\tau} \right\|_{\infty} \leq c \left[ \MESHRATIO( X_N ) \right]^d N^{-1} \sum_{k=m+1}^\infty ( 2^k )^{d - (s - \tau)} \\
&\leq c^{iv} \left[ \MESHRATIO( X_N ) \right]^d  N^{-(s-\tau) / d}.
\end{align*}
This completes the proof.
\end{proof}

We are now ready to prove Lemma~\ref{lem:Bernstein.inequality}.

\begin{proof}[Proof of Lemma~\ref{lem:Bernstein.inequality}]
For $N \geq 1$, let $m = \lfloor \frac{1}{d} \log_2 N \rfloor$.
First, observe from \eqref{eq:auxiliary.estimates} that $\| \BeKeNumIntErr{N}{s} \|_q$ is positive.
Hence, by \eqref{eq:strategy.inequality} and \eqref{eq:strategy.inequality.polynomial.part},
\begin{equation*}
\left\| \BeKeNumIntErr{N}{s} \right\|_{\sob{q}{\tau}{d}} \leq \left( c_{q,\tau,d} \left\| \eta_m \right\|_{1,d} N^{\tau/d} + \frac{\left\| \BeKeNumIntErr{N}{s} - \eta_m \ast \BeKeNumIntErr{N}{s} \right\|_{\sob{q}{\tau}{d}}}{\left\| \BeKeNumIntErr{N}{s} \right\|_q} \right) \left\| \BeKeNumIntErr{N}{s} \right\|_q.
\end{equation*}
By Lemma~\ref{lem:filtered.convolution.approx.error} and \eqref{eq:auxiliary.estimates}, the ratio is upper bounded by
\begin{equation*}
\left\| \BeKeNumIntErr{N}{s} - \eta_m \ast \BeKeNumIntErr{N}{s} \right\|_{\sob{q}{\tau}{d}} \Big/ \left\| \BeKeNumIntErr{N}{s} \right\|_q \leq \frac{c \left[ \MESHRATIO( X_N ) \right]^{d/p} N^{(\tau-s)/d}}{c_{p,s,d}^\prime \, N^{-s/d}} = c^{\prime\prime} \left[ \MESHRATIO( X_N ) \right]^{d/p} N^{\tau/d}.
\end{equation*}
Therefore
\begin{equation*}
\left\| \BeKeNumIntErr{N}{s} \right\|_{\sob{q}{\tau}{d}} \leq \left( c_{q,s,d} \left\| \eta_m \right\|_{1,d} + c^{\prime\prime} \left[ \MESHRATIO( X_N ) \right]^{d/p} \right) N^{\tau/d} \left\| \BeKeNumIntErr{N}{s} \right\|_q.
\end{equation*}
The result follows by observing with the aid of Lemma~\ref{lem:q.norm.boundedness.Bessel.kernel} that $\left\| \eta_m \right\|_{1,d}$ is bounded uniformly in $m$.
\end{proof}

\section{Example: The unit circle}
\label{sec:unit.circle}

In order to gain more insight into the covering problem, we turn to the unit circle $\sph{}\DEF\sph{1}$ and exploit the fact that this one-dimensional manifold is more accessible than its higher-dimensional counterparts and appeal at the same time to the general principle that certain fundamental features are shared across changing dimensions.
Circular designs (i.e. equally spaced points) on $\sph{}$ are exact for all trigonometric polynomials with degree strictly less than the number of points. They form QMC-design sequences, that is, give rise to optimal order worst-case error for QMC methods that integrate functions from the Sobolev space $\sob{p}{s}{}$ for every $p \geq 1$ and every $s > 1/p$. One question then is: \emph{How much of the QMC-design property is destroyed when just one point is removed from each configuration?}

We interpret $\sob{p}{s}{}$ as Bessel potential space (see Section~\ref{sec:function.space.setting}).\footnote{Alternatively, for $p = 2$ one can use the approach in \cite{BrSaSlWo2014}.} The Bessel kernel for $\sph{}$ then reduces to the Fourier cosine series
\begin{equation} \label{eq:BeKe.unit.circle}
\BeKeMod{s}( \cos \phi ) = \BeKe{s}( \cos \phi ) - 1 = 2 \sum_{\ell=1}^\infty \frac{\cos( \ell \phi )}{( 1 + \ell^2 )^{s/2}}.
\end{equation}
The $\Lp{q}{}$-norm of $\BeKe{s}$ is bounded if $s > 1/p$ with $1/q + 1/p = 1$ (Lemma~\ref{lem:q.norm.boundedness.Bessel.kernel}).
The worst-case error of $\numint[X_N]$ of a node set $X_N \subset \sph{}$ can then be expressed in terms of an appropriate Bessel kernel (see Theorem~\ref{thm:wce.for.Wp} and the remark following this theorem).
For the asymptotic analysis of the worst-case error we express the Bessel kernel in terms of generalized Clausen functions; i.e., \footnote{We use the \emph{Pochhammer symbol} to denote rising factorials: $\Pochhsymb{a}{0} \DEF 1$, $\Pochhsymb{a}{n+1} \DEF (n + a) \Pochhsymb{a}{n}$.}
\begin{equation} \label{eq:BeKeMod.unit.circle.Clausen.fcn}
\BeKeMod{s}( \cos \phi ) = 2 \ClausenCi_{s}( \phi ) + 2 \sum_{m=1}^\infty (-1)^m \frac{\Pochhsymb{s/2}{m}}{m!} \ClausenCi_{s+2m}( \phi ),
\end{equation}
where for $\re z > 1$ the \emph{generalized Clausen cosine} and \emph{sine functions} are defined as
\begin{equation*}
\ClausenCi_z( \phi ) \DEF \sum_{\ell=1}^\infty \frac{\cos( \ell \, \phi )}{\ell^z}, \qquad \ClausenSi_z( \phi ) \DEF \sum_{\ell=1}^\infty \frac{\sin( \ell \, \phi )}{\ell^z}
\end{equation*}
which may be extended to the complex $z$-plane by analytic continuation.

\begin{remark*}
By mapping the unit circle to the interval $[0,1)$, the functions in $\sob{p}{s}{}$ become Fourier series
\begin{equation*}
f(x) = \sum_{k \in \mathbb{Z}} \widehat{f}(k) \, \mathrm{e}^{2\pi \mathrm{i} k x}.
\end{equation*}
In the Hilbert space setting ($p = 2$), a slight modification of the coefficients in \eqref{eq:BeKe.unit.circle} ($(1+\ell^2)^{-s/2}$ is changed to
$r_s( 0 ) \DEF 1$ and $r_s( \ell ) \DEF |\ell|^{-s}$ for $\ell \geq 1$),
gives the standard Korobov space \cite{SW01}, which is a reproducing kernel Hilbert space with reproducing kernel
\begin{equation*}
K_s(x,y) = \sum_{\ell \in \mathbb{Z}} r_s(\ell) \, \mathrm{e}^{2 \pi \mathrm{i} \ell (x-y)} = 1 + 2 \sum_{\ell = 1}^\infty \frac{\cos (2\pi \ell (x-y))}{\ell^s}.
\end{equation*}
Since $\ell^s \le (1+\ell^2)^{s/2} \le 2^{s/2} \ell^s$ for $\ell \ge 1$, we have that the change of coefficients yields a space with equivalent norm. Numerical integration in (tensor-product) Korobov spaces is discussed in many papers, see \cite[Section~5]{DKS13}.
\end{remark*}

It is natural to study the Hilbert space setting (when $p = 2$) and the general non-Hilbert space setting (when $p \geq 1$), separately.

\subsection{Hilbert space setting}

As described in \cite{BrSaSlWo2014}, the \emph{strength} (more precisely, the \emph{$2$-strength}) of a sequence $(X_N)$ of $N$-point sets on $\sph{}$ is the supremum of the indices $s \geq 1/2$ for which $(X_N)$ is a QMC-design sequence for $\sob{2}{s}{}$. In particular, the $2$-strength of a sequence of circular designs $X_N$ with $N$ equally spaced points as $N \to \infty$ is infinite.

\begin{theorem} \label{thm:one.point.removed}
Let $s > 1/2$. A sequence of configurations of $N$ equally spaced points after one point is removed (or a uniformly bounded number of points are removed) from each configuration is a QMC-design sequence for $\sob{2}{s}{}$ for every $1/2 < s \leq 1$ but not for $s > 1$; i.e., such a sequence has $2$-strength $1$.
\end{theorem}

\begin{proof}
Let $s > 1/2$. Then the Bessel kernel $\BeKe{2s}$ is a reproducing kernel for the Bessel potential space $\sob{2}{s}{}$ and by the reproducing kernel Hilbert space approach the squared worst-case error of $\numint[X_N]$ of a node set $X_N = \{ ( \cos \phi_j, \sin \phi_j ) \}_{j=0}^{N-1} \subset \sph{}$ has the form
\begin{equation} \label{eq:general.wce.formula.unit.cirlce.p.EQ.2}
\left[ \WCE( \numint[X_N]; \sob{2}{s}{} ) \right]^2 = \frac{1}{N^2} \sum_{j=0}^{N-1} \sum_{k=0}^{N-1} \BeKeMod{2s}( \cos( \phi_j - \phi_k ) ).
\end{equation}

Let the points in $X_N$ be the equally spaced $N$th roots of unity so that $\phi_j = 2 \pi j / N$, $j = 0, \dots, N-1$. Such points are circular $(N-1)$-designs and satisfy the following identities: let $\ell = 0, 1, 2, \ldots$, then
\begin{equation} \label{eq:trig.identities}
\sum_{k=0}^{N-1} \sin \frac{2 \pi \ell k}{N} = 0, \qquad \sum_{k=0}^{N-1} \cos \frac{2 \pi \ell k}{N} = \begin{cases} N & \text{if $N \mid \ell$,} \\ 0 & \text{if $N \nmid \ell$,} \end{cases}
\end{equation}
where $N \mid \ell$ means that $\ell$ is an integer multiple of $N$ (``$N$ divides $\ell$'') and $N \nmid \ell$ means that $\ell$ is not divisible by $N$.
Substituting \eqref{eq:BeKeMod.unit.circle.Clausen.fcn} into the worst-case error formula \eqref{eq:general.wce.formula.unit.cirlce.p.EQ.2} and using \eqref{eq:trig.identities}, straightforward computation gives
\begin{equation}
\WCE( \numint[X_N]; \sob{2}{s}{} ) = \frac{1}{N^s} \left( 2 \zetafcn( 2s ) + \sum_{m=1}^\infty (-1)^m \frac{\Pochhsymb{s}{m}}{m!} \frac{2 \zetafcn( 2s + 2m )}{N^{2m}} \right)^{1/2};
\end{equation}
i.e., we recover the fact that equally spaced points, indeed, form QMC-design sequences for~$\sob{2}{s}{}$ for each $s > 1/2$.

Now, let $Z_{N-M}$ denote the collection of $N$th roots of unity with the first $M$ points omitted.
Using \eqref{eq:trig.identities}, it is readily verified that
\begin{equation} \label{eq:inner.sum}
\sum_{k = M}^{N-1} \BeKeMod{2s}\big( \cos\big( \frac{2\pi k}{N} - \phi \big) \big) = 2 N \sum_{\nu=1}^\infty \frac{\cos( N \phi )}{( 1 + \nu^2 \, N^2 )^s} - \sum_{k=0}^{M-1} \BeKeMod{2s}\big( \cos\big( \frac{2\pi k}{N} - \phi \big) \big).
\end{equation}
Substituting into the worst-case error formula \eqref{eq:general.wce.formula.unit.cirlce.p.EQ.2}, we get
\begin{align*}
&\left[ \WCE( \numint[Z_{N-M}]; \sob{2}{s}{} ) \right]^2
= \frac{1}{( N - M )^2} \sum_{j=M}^{N-1} \sum_{k=M}^{N-1} \BeKeMod{2s}\big( \cos\big( \frac{2\pi k}{N} - \frac{2\pi j}{N} \big) \big) \\
&\phantom{equals}= \frac{2 N \left( N - M \right)}{( N - M )^2} \sum_{\nu=1}^\infty \frac{1}{( 1 + \nu^2 \, N^2 )^s} - \frac{1}{(N - M)^2} \sum_{k=0}^{M-1} \sum_{j=M}^{N-1} \BeKeMod{2s}\big( \cos\big( \frac{2\pi j}{N} - \frac{2\pi k}{N} \big) \big).
\end{align*}
A second application of \eqref{eq:inner.sum} gives
\begin{equation*}
\left[ \WCE( \numint[Z_{N-M}]; \sob{2}{s}{} ) \right]^2 = \frac{2 N \left( N - 2 M \right)}{( N - M )^2} \sum_{\nu=1}^\infty \frac{1}{( 1 + \nu^2 \, N^2 )^s} + \frac{M^2}{(N - M)^2} \, \BeKeNumIntErr{N,M}{2s},
\end{equation*}
where
\begin{align*}
\BeKeNumIntErr{N,M}{2s}
&\DEF \frac{1}{M^2} \sum_{j=0}^{M-1} \sum_{k=0}^{M-1} \BeKeMod{2s}\big( \cos \frac{2\pi ( j - k )}{N} \big) \\
&= \BeKeMod{2s}( 1 ) - \frac{2}{M^2} \, \sum_{\nu=1}^{M-1} ( M - \nu ) \left[ \BeKeMod{2s}( 1 ) - \BeKeMod{2s}\big( \cos \frac{2\pi \nu}{N} \big) \right]
\end{align*}
and
\begin{equation} \label{eq:BeKeMod.at.1}
\BeKeMod{2s}( 1 ) = 2 \sum_{m=0}^\infty (-1)^m \frac{\Pochhsymb{s}{m}}{m!} \zetafcn( 2s + 2m ) = 2 \sum_{\ell=1}^\infty \frac{1}{( 1 + \ell^2 )^{s/2}}.
\end{equation}
Observe from \eqref{eq:BeKe.unit.circle} that the square-bracketed expression above is non-negative. Furthermore, one has
\begin{equation*}
0 \leq \BeKeMod{2s}( 1 ) - \BeKeNumIntErr{N,M}{2s}  \leq \left( \frac{2}{M^2} \, \sum_{\nu=1}^{M-1} ( M - \nu ) \right) \max_{0 \leq x \leq 2\pi M / N} \left[ \BeKeMod{2s}( 1 ) - \BeKeMod{2s}\big( \cos \frac{2\pi \nu}{N} \big) \right].
\end{equation*}
Since the parenthetical expression is bounded by $1$ and the maximum tends to zero when $M / N \to 0$, we arrive at
\begin{equation*}
\left[ \WCE( \numint[Z_{N-M}]; \sob{2}{s}{} ) \right]^2 = \frac{2 N \left( N - 2 M \right)}{( N - M )^2} \sum_{\nu=1}^\infty \frac{1}{( 1 + \nu^2 \, N^2 )^s} + \frac{M^2}{(N - M)^2} \left\{ \BeKeMod{2s}( 1 ) + o( 1 ) \right\}.
\end{equation*}
Rewriting the infinite series, we finally arrive at
\begin{equation*}
\left[ \WCE( \numint[Z_{N-M}]; \sob{2}{s}{} ) \right]^2 = \frac{M^2}{( N - M )^2} \left\{ \BeKeMod{2s}(1) + \cdots \right\} + \frac{ N ( N - 2 M )}{( N - M )^2} \left\{ \frac{2 \zetafcn( 2s )}{N^{2s}} + \cdots \right\}
\end{equation*}
Let $M = 1$. Then for $1/2 < s < 1$ we obtain the asymptotics
\begin{equation*}
\WCE( \numint[Z_{N-1}]; \sob{2}{s}{} ) = \frac{\sqrt{2 \zetafcn( 2s )}}{( N - 1 )^s} \left\{ 1 + \frac{\BeKeMod{2s}( 1 )}{4 \zetafcn( 2s )} \left( N - 1 \right)^{2s-2} + \cdots \right\},
\end{equation*}
whereas for $s > 1$ we have that
\begin{equation*}
\WCE( \numint[Z_{N-1}]; \sob{2}{s}{} ) = \frac{\sqrt{\BeKeMod{2s}( 1 )}}{N - 1} \left\{ 1 + \frac{\zetafcn( 2s )}{\BeKeMod{2s}( 1 )} \left( N - 1 \right)^{2-2s} + \cdots \right\}.
\end{equation*}
We conclude that $( Z_{N-1} )$ is a QMC-design sequence for $\sob{2}{s}{}$ if and only if $1/2 < s \leq 1$; i.e., the $2$-strength of $( Z_{N-1} )$ is $1$.
This completes the proof when one point is omitted.
\end{proof}
A similar but more tedious argument provides the leading term behavior of the asymptotics of the worst-case error when a finite number (uniformly upper bounded) of points are removed from a circular design.

Let $Z_{N-M}$ denote a configuration of $N$ equally spaced points on $\sph{}$ with $M$ consecutive points removed. The hole thus generated in $Z_{N-M}$ has covering radius
\begin{equation} \label{eq:rho.Z.N.minus.M}
\COV( Z_{N-M} ) = \frac{\pi ( M + 1 )}{N}.
\end{equation}
(Note that $\pi / N$ is the packing radius of the $N$ equally sized circular arcs making up $\sph{}$ which is half of the minimal geodesic separation distance of points in $Z_{N-M}$.) For our discussion we want to assume that the hole size shrinks as $N$ grows, so $M / N \to 0$ as $N \to \infty$.
Theorem~\ref{thm:one.point.removed} covers the case when $M$ is uniformly bounded. We now consider the case $M \to \infty$ which is equivalent to $N \COV( Z_{N - M} ) \to \infty$ as $N \to \infty$. Then the sequence $( Z_{N-M} )$ does not have the optimal covering property. The next theorem shows that, despite bad covering, $(Z_{N-M})$ has $2$-strength $1$ if the artificially generated holes shrink rapidly enough.
Interestingly, the sequence $(Z_{N-M})$ is {\bf not} a QMC-design sequence for $\sob{2}{s}{}$ for~$s = 1$ .
Moreover, if the hole size shrinks too slowly, then $( Z_{N - M} )$ is not a QMC-design sequence for any $s > 1/2$.
However, one can choose the asymptotic behavior of the covering radius to get as close as one likes to a QMC-design sequence for $\sob{2}{s}{}$ for $1/2 < s < 1$ (e.g., when the covering radius behaves like $(\log \circ \cdots \circ \log N) / N^s$).

\begin{theorem} \label{thm:M.points.removed}
Let $s > 1/2$ and $( Z_{N-M} )$ be as above with $M \to \infty$ and $M / N \to 0$. 
\begin{enumerate}[(a)]
\item If $N^{s} \rho( Z_{N - M} ) \to c$ for some real $c \geq 0$, then $( Z_{N-M} )$ has $2$-strength~$1$ but is not a QMC-design sequence for $s = 1$.
\item If $N^{s} \rho( Z_{N - M} ) \to \infty$, then $( Z_{N-M} )$ is not a QMC-design sequence for $\sob{2}{s}{}$ with $s > 1/2$. In particular, when $1/2 < s < 3/2$,
\begin{equation*}
\WCE( \numint[Z_{N-M}]; \sob{2}{s}{} ) = \sqrt{\BeKeMod{2s}(1)} \, \frac{\rho( Z_{N - M} )}{\pi} \left\{ 1 + o( 1 ) \right\} \qquad \text{as $N \to \infty$,}
\end{equation*}
where $\BeKeMod{2s}(1)$ is the constant in \eqref{eq:BeKeMod.at.1}.
\end{enumerate}
\end{theorem}

\begin{proof}
We proceed along the same lines as the proof of Theorem~\ref{thm:one.point.removed} and determine the asymptotic (large $N$) behavior of the worst-case error for QMC methods based on node sets $Z_{N-M}$ for functions in $\sob{2}{s}{}$ with $s > 1/2$.
Let $Z_{N-M}$ denote the collection of $N$th roots of unity with the first $M$ points omitted.
For $M / N \to 0$ as $N \to \infty$, we obtained the following asymptotics in the proof of Theorem~\ref{thm:one.point.removed}:
\begin{equation*}
\left[ \WCE( \numint[Z_{N-M}]; \sob{2}{s}{} ) \right]^2 = \frac{M^2}{( N - M )^2} \left\{ \BeKeMod{2s}(1) + \cdots \right\} + \frac{ N ( N - 2 M )}{( N - M )^2} \left\{ \frac{2 \zetafcn( 2s )}{N^{2s}} + \cdots \right\}.
\end{equation*}
Now, let $M$ grow with $N$ such that $M / N \to 0$ and $M \to \infty$ as $N \to \infty$. The unboundedness of $M$ implies that $N \COV( Z_{N-M} ) = \pi ( M + 1 ) \to \infty$ as $N \to \infty$ and thus $( Z_{N - M} )$ does not have the optimal covering property. However, for $N^s \, M / ( N - M ) \to c$ as $N \to \infty$ for some real $c \geq 0$, we still have that
\begin{equation*}
\WCE( \numint[Z_{N-M}]; \sob{2}{s}{} ) = \frac{\left( 2 \zetafcn( 2s ) + c^2 \BeKeMod{2s}(1) \right)^{1/2}}{N^{s}} \left\{ 1 + o( 1 ) \right\} \qquad \text{as $N \to \infty$,}
\end{equation*}
where $\BeKeMod{2s}(1)$ is given in \eqref{eq:BeKeMod.at.1}. Thus $( Z_{N - M} )$ is a QMC-design sequence for $\sob{2}{s}{}$ for ${1/2 < s < 1}$. (The upper bound on $s$ is imposed by the unboundedness of $M$.) On the other hand, when $N^s \, M / ( N - M ) \to \infty$, then
\begin{equation*}
\WCE( \numint[Z_{N-M}]; \sob{2}{s}{} ) = \sqrt{\BeKeMod{2s}(1)} \, \frac{M}{N - M} \left\{ 1 + o( 1 ) \right\} \qquad \text{as $N \to \infty$.}
\end{equation*}
The last convergence relation for $M$ is automatically satisfied for $s \geq 1$ and gives suboptimal convergence rate for the worst-case error when $1/2 < s < 3/2$. The result follows by using the covering radius instead of $M$ (see~\eqref{eq:rho.Z.N.minus.M}).
\end{proof}

\subsection{The general case $p \geq 1$}
We now leave the Hilbert space setting and consider $\sob{p}{s}{}$ for $p \geq 1$.
Let $s > 1/p$. By Theorem~\ref{thm:wce.for.Wp} the worst-case error for $\numint[Z_N]$ for a circular design consisting of the $N$th roots of unity for $\sob{p}{s}{}$ is given by the $\Lp{q}{}$-norm of the function $\BeKeNumIntErr{N}{s}$ of \eqref{eq:SBF-Bessel}. On the unit circle one can write with the help of \eqref{eq:inner.sum},
\begin{equation*}
\BeKeNumIntErr{N}{s}( \cos \phi ) = \frac{1}{N} \sum_{k = 0}^{N-1} \BeKeMod{s}\big( \cos\big( \frac{2\pi k}{N} - \phi \big) \big) = \frac{2}{N^s} \sum_{\nu=1}^\infty \frac{\cos( N \phi )}{( \nu^2 + 1/N^2 )^{s/2}}.
\end{equation*}
Hence
\begin{equation*}
\WCE( \numint[Z_{N}]; \sob{p}{s}{} ) = \left\| \BeKeNumIntErr{N}{s} \right\|_q = \frac{2}{N^{s}} \left( \frac{1}{2\pi} \int_0^{2\pi} \left| \sum_{\nu=1}^\infty \frac{\cos( N \phi )}{( \nu^2 + 1/N^2 )^{s/2}} \right|^q \dd \phi \right)^{1/q}.
\end{equation*}
Dividing the integration domain into $N$ parts and using the $2\pi$-periodicity of the integrand, it follows that
\begin{equation*}
\WCE( \numint[Z_{N}]; \sob{p}{s}{} ) = \frac{2}{N^{s}} \left( \frac{1}{2\pi} \int_0^{2\pi} \left| \sum_{\nu=1}^\infty \frac{\cos \phi }{( \nu^2 + 1/N^2 )^{s/2}} \right|^q \dd \phi \right)^{1/q}.
\end{equation*}
For large $N$, the series can be approximated by the generalized Clausen cosine function. A mean value argument ($( x^2 + \varepsilon )^{q/2} = | x |^q + \frac{1}{2} q \varepsilon ( x^2 + \varepsilon^\prime )^{q/2-1}$ for $0 < \varepsilon^\prime < \varepsilon$) gives that
\begin{equation} \label{eq:wce.asymptotics.p.GEQ.1}
\WCE( \numint[Z_{N}]; \sob{p}{s}{} ) = \frac{2}{N^{s}} \left( \frac{1}{2\pi} \int_0^{2\pi} \left| \ClausenCi_s( \phi ) \right|^q \dd \phi + \mathcal{O}(N^{-2} ) \right)^{1/q} \qquad \text{as $N \to \infty$.}
\end{equation}

Let $Z_{N-M}$ be the set of $N$th roots of unity with $M$ consecutive points omitted. Then \eqref{eq:inner.sum} gives  that
\begin{align*}
\BeKeNumIntErr{N-M}{s}( \cos \phi )
&= \frac{1}{N-M} \sum_{k = 1}^{N-M} \BeKeMod{s}\big( \cos\big( \frac{2\pi k}{N} - \phi \big) \big) \\
&= \frac{N}{N-M} \, \frac{2}{N^s} \sum_{\nu=1}^\infty \frac{\cos( N \phi )}{( \nu^2 + 1/N^2 )^{s/2}} - \frac{1}{N-M} \sum_{k=0}^{M-1} \BeKeMod{s}\big( \cos\big( \frac{2\pi k}{N} - \phi \big) \big). 
\end{align*}
Similarly as before, we get
\begin{equation*}
\begin{split}
\WCE( \numint[Z_{N-M}]; \sob{p}{s}{} )
&= \Bigg( \frac{1}{N} \sum_{k=0}^{N-1} \int_0^1 \Bigg| \frac{N}{N-M} \, \frac{2}{N^s} \sum_{\nu=1}^\infty \frac{\cos( 2 \pi x )}{( \nu^2 + 1/N^2 )^{s/2}} \\
&\phantom{equals}- \frac{1}{N-M} \sum_{j=0}^{M-1} \BeKeMod{s}\big( \cos \frac{2\pi( x + k - j )}{N} \big) \Bigg|^q \dd x \Bigg)^{1/q}.
\end{split}
\end{equation*}
Approximation with a generalized Clausen cosine function gives
\begin{equation*}
\begin{split}
\WCE( \numint[Z_{N-M}]; \sob{p}{s}{} )
&= \Bigg( \frac{1}{N} \sum_{k=0}^{N-1} \int_0^1 \Bigg| \frac{N}{N-M} \, \frac{2 \ClausenCi_s( 2\pi x ) + \mathcal{O}(N^{-2})}{N^s} \\
&\phantom{equals}- \frac{M}{N-M} \, \frac{1}{M} \sum_{j=0}^{M-1} \BeKeMod{s}\big( \cos \frac{2\pi( x + k - j )}{N} \big) \Bigg|^q \dd x \Bigg)^{1/q}.
\end{split}
\end{equation*}
The asymptotic behavior of the worst-case error is determined by the limiting behavior of $N^s \COV( Z_{N-M} )$ as $N \to \infty$. Similar results to the Hilbert space setting can be derived. We leave this to the reader. (Particular care is needed when both contributions between the absolute value signs are in ``balance'' for large $N$; e.g., when $M = 1$ and $s = 1$.)

\vspace{1em}

{\bf Acknowledgements:} The authors are grateful to an anonymous referee for valuable comments that improved the paper.

\bibliographystyle{abbrv}
\bibliography{bibliography}

\begin{thebibliography}{10}

\bibitem{Au:1998}
T.~Aubin.
\newblock {\em Some nonlinear problems in {R}iemannian geometry}.
\newblock Springer Monographs in Mathematics. Springer-Verlag, Berlin, 1998.

\bibitem{BeBuPa1968}
H.~Berens, P.~L. Butzer, and S.~Pawelke.
\newblock Limitierungsverfahren von {R}eihen mehrdimensionaler
  {K}ugelfunktionen und deren {S}aturationsverhalten.
\newblock {\em Publ. Res. Inst. Math. Sci. Ser. A}, 4:201--268, 1968/1969.

\bibitem{BeLo1976}
J.~Bergh and J.~L{\"o}fstr{\"o}m.
\newblock {\em Interpolation spaces. {A}n introduction}.
\newblock Springer-Verlag, Berlin, 1976.
\newblock Grundlehren der Mathematischen Wissenschaften, No. 223.

\bibitem{BoRaVi2013}
A.~Bondarenko, D.~Radchenko, and M.~Viazovska.
\newblock Optimal asymptotic bounds for spherical designs.
\newblock {\em Ann. of Math. (2)}, 178(2):443--452, 2013.

\bibitem{BoRaVi2014}
A.~Bondarenko, D.~Radchenko, and M.~Viazovska.
\newblock Well-separated spherical designs.
\newblock {\em Constr. Approx.}, pages 1--20 (published online), 2014.

\bibitem{BCCGST2013}
L.~Brandolini, C.~Choirat, L.~Colzani, G.~Gigante, R.~Seri, and G.~Travaglini.
\newblock Quadrature rules and distribution of points on manifolds.
\newblock {\em Ann. Sc. Norm. Super. Pisa Cl. Sci.}, page~35, 2013 (in press).

\bibitem{Brauchart2008:optimal_logarithmic}
J.~S. Brauchart.
\newblock Optimal logarithmic energy points on the unit sphere.
\newblock {\em Math. Comp.}, 77(263):1599--1613, 2008.

\bibitem{BrHe2007}
J.~S. Brauchart and K.~Hesse.
\newblock Numerical integration over spheres of arbitrary dimension.
\newblock {\em Constr. Approx.}, 25(1):41--71, 2007.

\bibitem{BrSaSlWo2014}
J.~S. Brauchart, E.~B. Saff, I.~H. Sloan, and R.~S. Womersley.
\newblock {QMC designs: Optimal order Quasi Monte Carlo integration schemes on
  the sphere}.
\newblock {\em Math. Comp.}, 83(290):2821--2851, 2014.

\bibitem{DeGoSei1977}
P.~Delsarte, J.~M. Goethals, and J.~J. Seidel.
\newblock Spherical codes and designs.
\newblock {\em Geometriae Dedicata}, 6(3):363--388, 1977.

\bibitem{DKS13}
J.~Dick, F.~Y. Kuo, and I.~H. Sloan.
\newblock High-dimensional integration: the quasi-{M}onte {C}arlo way.
\newblock {\em Acta Numer.}, 22:133--288, 2013.

\bibitem{NIST:DLMF}
{NIST Digital Library of Mathematical Functions}.
\newblock http://dlmf.nist.gov/, Release 1.0.6 of 2013-05-06.
\newblock Online companion to \cite{Olver:2010:NHMF}.

\bibitem{He2006}
K.~Hesse.
\newblock A lower bound for the worst-case cubature error on spheres of
  arbitrary dimension.
\newblock {\em Numer. Math.}, 103(3):413--433, 2006.

\bibitem{HeSl2005b}
K.~Hesse and I.~H. Sloan.
\newblock Optimal lower bounds for cubature error on the sphere $\mathbb{S}^2$.
\newblock {\em J. Complexity}, 21(6):790--803, 2005.

\bibitem{HeSl2005a}
K.~Hesse and I.~H. Sloan.
\newblock Worst-case errors in a {S}obolev space setting for cubature over the
  sphere $\mathbb{S}^2$.
\newblock {\em Bull. Austral. Math. Soc.}, 71(1):81--105, 2005.

\bibitem{HuZae2009}
J.~Hu and M.~Z{\"a}hle.
\newblock Generalized {B}essel and {R}iesz potentials on metric measure spaces.
\newblock {\em Potential Anal.}, 30(4):315--340, 2009.

\bibitem{LeMh2006}
Q.~T. Le~Gia and H.~N. Mhaskar.
\newblock Polynomial operators and local approximation of solutions of
  pseudo-differential equations on the sphere.
\newblock {\em Numer. Math.}, 103(2):299--322, 2006.

\bibitem{MhNaPrWa2010}
H.~N. Mhaskar, F.~J. Narcowich, J.~Prestin, and J.~D. Ward.
\newblock {$L^p$} {B}ernstein estimates and approximation by spherical basis
  functions.
\newblock {\em Math. Comp.}, 79(271):1647--1679, 2010.

\bibitem{Mu1966}
C.~M{\"u}ller.
\newblock {\em Spherical harmonics}, volume~17 of {\em Lecture Notes in
  Mathematics}.
\newblock Springer-Verlag, Berlin, 1966.

\bibitem{NaPeWa2006}
F.~Narcowich, P.~Petrushev, and J.~Ward.
\newblock Decomposition of {B}esov and {T}riebel-{L}izorkin spaces on the
  sphere.
\newblock {\em J. Funct. Anal.}, 238(2):530--564, 2006.

\bibitem{NaPeWa2006-2}
F.~J. Narcowich, P.~Petrushev, and J.~D. Ward.
\newblock Localized tight frames on spheres.
\newblock {\em SIAM J. Math. Anal.}, 38(2):574--594 (electronic), 2006.

\bibitem{NaSuWa2010}
F.~J. Narcowich, X.~Sun, J.~D. Ward, and Z.~Wu.
\newblock Le{V}eque type inequalities and discrepancy estimates for minimal
  energy configurations on spheres.
\newblock {\em J. Approx. Theory}, 162(6):1256--1278, 2010.

\bibitem{Nie92}
H.~Niederreiter.
\newblock {\em {Random Number Generation and quasi-Monte Carlo Methods}}.
\newblock CBMS-NSF Regional Conference Series in Applied Mathematics. Society
  for Industrial and Applied Mathematics, Philadelphia, PA, USA, 1992.

\bibitem{Olver:2010:NHMF}
F.~W.~J. Olver, D.~W. Lozier, R.~F. Boisvert, and C.~W. Clark, editors.
\newblock {\em {NIST Handbook of Mathematical Functions}}.
\newblock Cambridge University Press, New York, NY, 2010.
\newblock Print companion to \cite{NIST:DLMF}.

\bibitem{Rei1999}
M.~Reimer.
\newblock Spherical polynomial approximations: a survey.
\newblock In {\em Advances in multivariate approximation
  ({W}itten-{B}ommerholz, 1998)}, volume 107 of {\em Math. Res.}, pages
  231--252. Wiley-VCH, Berlin, 1999.

\bibitem{Rei2000}
M.~Reimer.
\newblock Hyperinterpolation on the sphere at the minimal projection order.
\newblock {\em J. Approx. Theory}, 104(2):272--286, 2000.

\bibitem{Ru1992}
K.~P. Rustamov.
\newblock On the equivalence of the {$K$}-functional and the modulus of
  smoothness of functions on a sphere.
\newblock {\em Mat. Zametki}, 52(3):123--129, 160, 1992.

\bibitem{SW01}
I.~H. Sloan and H.~Wo{\'z}niakowski.
\newblock Tractability of multivariate integration for weighted {K}orobov
  classes.
\newblock {\em J. Complexity}, 17(4):697--721, 2001.
\newblock Complexity of multivariate problems (Kowloon, 1999).

\bibitem{Str1983}
R.~S. Strichartz.
\newblock Analysis of the {L}aplacian on the complete {R}iemannian manifold.
\newblock {\em J. Funct. Anal.}, 52(1):48--79, 1983.

\bibitem{Yud1995}
V.~A. Yudin.
\newblock Coverings of a sphere, and extremal properties of orthogonal
  polynomials.
\newblock {\em Diskret. Mat.}, 7(3):81--88, 1995.

\end{thebibliography}

\end{document}